\documentclass{amsart}
\usepackage{graphicx}
\usepackage{amsmath,amsthm,amssymb,enumerate}
\usepackage{euscript,mathrsfs}
\usepackage[citecolor=blue,colorlinks=true]{hyperref}
\usepackage[left=3cm,right=3cm,top=3.5cm,bottom=3.5cm]{geometry}
\usepackage{color}
\catcode`\@=11 \@addtoreset{equation}{section}

\catcode`\@=12

\allowdisplaybreaks

\newtheorem{Theorem}{Theorem}[section]
\newtheorem{Proposition}[Theorem]{Proposition}
\newtheorem{Lemma}[Theorem]{Lemma}
\newtheorem{Corollary}[Theorem]{Corollary}

\theoremstyle{definition}
\newtheorem{Definition}[Theorem]{Definition}

\newtheorem{Remark}[Theorem]{Remark}

\newcommand{\bTheorem}[1]{
\begin{Theorem} \label{T#1} }
\newcommand{\eT}{\end{Theorem}}

\newcommand{\bProposition}[1]{
\begin{Proposition} \label{P#1}}
\newcommand{\eP}{\end{Proposition}}

\newcommand{\bLemma}[1]{
\begin{Lemma} \label{L#1} }
\newcommand{\eL}{\end{Lemma}}

\newcommand{\bCorollary}[1]{
\begin{Corollary} \label{C#1} }
\newcommand{\eC}{\end{Corollary}}

\newcommand{\bRemark}[1]{
\begin{Remark} \label{R#1} }
\newcommand{\eR}{\end{Remark}}

\newcommand{\bDefinition}[1]{
\begin{Definition} \label{D#1} }
\newcommand{\eD}{\end{Definition}}

\newcommand{\dif}{\mathrm{d}}

\newcommand{\mf}{\mathscr{F}}

\newcommand{\prst}{\mathbb{P}}
\newcommand{\p}{\mathbb{P}}

\newcommand{\ind}{\mathbf{1}}
\newcommand{\mn}{\mathbb{N}}

\newcommand{\mt}{\mathbb{T}^3}

\newcommand{\bu}{\mathbf u}
\newcommand{\mtor}{\mathbb{T}^3}

\newcommand{\bfw}{\mathbf{w}}
\newcommand{\bfu}{\mathbf{u}}
\newcommand{\bfU}{\mathbf{U}}
\newcommand{\bfV}{\mathbf{V}}
\newcommand{\bfv}{\mathbf{v}}
\newcommand{\bfq}{\mathbf{q}}
\newcommand{\bfg}{\mathbf{g}}

\newcommand{\g}{\gamma}
\newcommand{\G}{\Gamma}
\newcommand{\s}{\sigma}

\newcommand{\ve}{\varepsilon}
\renewcommand{\a}{\alpha}

\renewcommand{\d}{\delta}
\renewcommand{\k}{\kappa}

\newcommand{\bFormula}[1]{
\begin{equation} \label{#1}}
\newcommand{\eF}{\end{equation}}

\newcommand{\Ov}[1]{\overline{#1}}

\newcommand{\DC}{C^\infty_c}

\newcommand{\vr}{\varrho}
\newcommand{\vre}{\vr_\ep}

\newcommand{\vue}{\vu_\ep}

\newcommand{\vu}{\vc{u}}
\newcommand{\vc}[1]{{\bf #1}}

\newcommand{\Div}{{\rm div}_x}
\newcommand{\Grad}{\nabla_x}

\newcommand{\tn}[1]{\mathbb{#1}}
\newcommand{\dx}{\,{\rm d} {x}}

\newcommand{\dt}{\,{\rm d} t }

\newcommand{\intO}[1]{\int_{\mathbb T^3} #1 \dx}

\newcommand{\D}{{\rm d}}
\newcommand{\ep}{\varepsilon}

\newcommand{\R}{\mathbb{R}}

\newcommand{\E}{\mathbb{E}}

\newcommand{\expe}[1]{ \mathbb{E} \left[ #1 \right] }
\newcommand{\bas}{\big( \Omega, \mathfrak{F}, ( \mathfrak{F}_t )_{t \geq 0} ,\prst \big)}

\definecolor{Cgrey}{rgb}{0.85,0.85,0.85}
\definecolor{Cblue}{rgb}{0.50,0.85,0.85}
\definecolor{Cred}{rgb}{1,0,0}
\definecolor{fancy}{rgb}{0.10,0.85,0.10}

\newcommand\Cbox[2]{%
    \newbox\contentbox%
    \newbox\bkgdbox%
    \setbox\contentbox\hbox to \hsize{%
        \vtop{
            \kern\columnsep
            \hbox to \hsize{%
                \kern\columnsep%
                \advance\hsize by -2\columnsep%
                \setlength{\textwidth}{\hsize}%
                \vbox{
                    \parskip=\baselineskip
                    \parindent=0bp
                    #2
                }%
                \kern\columnsep%
            }%
            \kern\columnsep%
        }%
    }%
    \setbox\bkgdbox\vbox{
        \color{#1}
        \hrule width  \wd\contentbox %
               height \ht\contentbox %
               depth  \dp\contentbox
        \color{black}
    }%
    \wd\bkgdbox=0bp%
    \vbox{\hbox to \hsize{\box\bkgdbox\box\contentbox}}%
    \vskip\baselineskip%
}

\date{}

\begin{document}


\title[Stationary solutions to the stochastic compressible Navier-Stokes system]{Stationary solutions to the compressible Navier-Stokes system driven by stochastic forces}

\author{Dominic Breit}
\address[D. Breit]{Department of Mathematics, Heriot-Watt University, Riccarton Edinburgh EH14 4AS, UK}

\author{Eduard Feireisl}
\address[E. Feireisl]{Institute of Mathematics of the Academy of Sciences of the Czech Republic, \v Zitn\' a 25, CZ-115 67 Praha 1, Czech Republic}


\author{Martina Hofmanov\' a}
\address[M. Hofmanov\'a]{Technical University Berlin, Institute of Mathematics, Stra\ss e des 17. Juni 136, 10623 Berlin, Germany}

\author{Bohdan Maslowski}
\address[B. Maslowski]{Charles University in Prague, Faculty of Mathematics and Physics, Sokolovsk\' a 83, 186 75 Praha 8, Czech Republic}

\thanks{The research of E.F.~leading to these results has received funding from the
European Research Council under the European Union's Seventh
Framework Programme (FP7/2007-2013)/ ERC Grant Agreement
320078. B.M. has been supported by GACR grant no. 15-08819S. The Institute of Mathematics of the Academy of Sciences of
the Czech Republic is supported by RVO:67985840.}

\begin{abstract}

We study the long-time behavior of solutions to a stochastically driven Navier-Stokes system describing the motion of a compressible viscous fluid
driven by a temporal multiplicative white noise perturbation.
The existence of stationary solutions is established in the framework of Lebesgue--Sobolev spaces pertinent to the class of weak martingale solutions.
The methods are based on new global-in-time estimates and a combination of deterministic and stochastic compactness arguments. 
In contrast with the deterministic case, where related results were obtained only under rather restrictive constitutive assumptions
for the pressure, the stochastic case is tractable in the full range of constitutive relations allowed by the available existence theory. This can be seen 
as a kind of regularizing effect of the noise on the global-in-time solutions.
\end{abstract}

\keywords{Navier-Stokes system, compressible fluid, stochastic perturbation, stationary solution}

\date{\today}

\maketitle


\section{Introduction}
\label{i}

Stationary solutions of an evolutionary system provide an important piece of information concerning the behavior in the long run.
For systems with background in classical fluid mechanics, stationary solutions typically minimize the entropy production and play the role
of an attractor, at least for energetically insulated fluid flows, see e.g. \cite{FeiPr}.

{The principal question arising in the context of randomly driven systems} is the existence of a (stochastic) steady state solution for the system. Earlier results in this direction concern the incompressible case: Flandoli \cite{Fl1} proved existence of an invariant measure by the ``remote start'' method in 2D case.  This result has been extended in a few papers, for instance in Goldys and Maslowski \cite{GoMa1}, \cite{GoMa2} where existence of invariant measure by  the method of embedded Markov chain theory verifying also the exponential speed of convergence to invariant measure.  A different approach has been adopted by Hairer and Mattingly \cite{HM06} in which case a slightly weaker convergence result (implying however the uniqueness of invariant measure) has been shown under much weaker conditions on the nondegeneracy of the noise.
In the  paper \cite{BMO}  the existence of invariant measure is proved for 2D Navier-Stokes equation on unbounded domain  by a compactness method in the (weak) bw-topology.

In 3D much less is known in the case of incompressible fluid.  The problems here appear already on the level of Markov property induced by the equation. Transition Markov semigroup has been constructed in the papers by Da Prato and Debussche \cite{DaPDe1}, \cite{DaPDe2}, provided the noise term is sufficiently rough in space. However, even the problem of uniqueness of transition Markov semigroup remains open. A different approach was adopted by Flandoli and Romito \cite{FlaRom} who used the classical Stroock-Varadhan type argument to find a suitable Markov selection and construct the semigroup. The transition semigroup is shown to be  be exponentially ergodic (under appropriate conditions on the noise term) by the same arguments as in  \cite{GoMa1}. However, the uniqueness of the Markov transition semigroup has not been proved so far.

In absence of the Markov property (i.e. in the situation when the concept of invariant measure as a steady state is not well defined) it is possible to work directly with stationary solutions, i.e. with solutions which are strictly stationary stochastic processes. In the pioneering paper Flandoli and Gatarek \cite{FlaGat} existence of such  stationary solution has been shown in the 3D incompressible case by means of finite-dimensional approximations.

{To our best knowledge, no relevant results on large time behavior have been achieved so far for compressible stochastic fluid flows, 
where the situation is much more complex. 
To fill at least partially this gap,} we examine the class of stationary
solutions for a stochastically driven Navier-Stokes system:

\Cbox{Cgrey}{

\begin{eqnarray}
\label{i1}
   \D \vr + \Div (\vr \vu) \dt &=& 0, \\
\label{i2}
\D (\vr \vu) + \Div (\vr \vu \otimes \vu) \dt + \Grad p(\vr) \dt &=& \Div \mathbb{S}(\Grad \vu) \dt + \mathbb{G}(\vr, \vr \vu) \ \D {W}, \\
\label{i3}
\mathbb{S} (\Grad \vu) &=& \mu \left( \Grad \vu + \Grad^t \vu - \frac{2}{3} \Div \vu \mathbb{I} \right) + \eta \Div \vu \mathbb{I},
\end{eqnarray}

}

\noindent
where $\vr = \vr(t,x)$ is the mass density and $\vu = \vu(t,x)$ the macroscopic velocity of a compressible viscous fluids contained in a physical domain
$\mathcal O \subset \R^3$. Here the symbol $p=p(\vr)$ denotes the pressure, typically given by the isentropic state equation
\begin{equation} \label{i3a}
p(\vr) = a \vr^\gamma, \ a > 0,
\end{equation}
and $\mathbb{S}$ is the viscous stress tensor determined by Newton's rheological law (\ref{i3}), with  viscosity coefficients $\mu > 0$, $\eta \geq 0$.
The stochastic driving force is represented by the stochastic differential of the form
\begin{equation*} \label{i3b}
\mathbb{G}(\varrho,\varrho\bu)\, \D W = \sum_{k = 1}^\infty \vc{G}_k(x,\varrho,\varrho\bu)\,  \D W_k,
\end{equation*}
where $W=(W_k)_{k\in\mn}$ is a cylindrical Wiener process specified in Section \ref{s} below.

As pointed out above,
in contrast with the frequently studied \emph{incompressible} Navier-Stokes system, the problems related to the dynamics of
\emph{compressible} fluid flows driven by stochastic forcing are basically open.
First existence results were based on a suitable transformation formula that allows to reduce the problem to a random system of PDEs: The stochastic integral does no longer appear and deterministic methods are applicable, see \cite{MR1760377} for the 1D case, \cite{MR1807944} for a rather special periodic 2D. Finally, the work by Feireisl, Maslowski, Novotn\'y \cite{MR2997374} deals with the 3D case.
The first ``truly'' stochastic existence result for the compressible Navier-Stokes system perturbed by a general nonlinear multiplicative noise was obtained by Breit, Hofmanov\'a \cite{BrHo}. The existence of the so-called finite energy weak martingale solutions in three space dimensions with periodic boundary conditions was established. Extension of this result to the zero Dirichlet boundary conditions then appeared in \cite{2015arXiv150400951S,MR3385137}. For completeness, let us also mention \cite{BrFeHo2015,BrFeHo2015A,BrFeHo2016}  where further results appeared, namely, a singular limit, the so-called relative energy inequality and the local existence of strong solutions, respectively.

{Our goal is to establish the existence of global--in--time solutions to system (\ref{i1})--(\ref{i3}) that are stationary in the stochastic sense. 
To this end, we use a \emph{direct} method based on the four layer approximation scheme developed in
\cite{BrHo} inspired by \cite{feireisl1}. More specifically, the stationary solutions are constructed at the very basic approximation level. The final result is obtained by means of a combination of deterministic and stochastic compactness methods.} 
To be more precise,  the equations are regularized {by adding artificial viscosity and an artificial pressure term to the momentum equation (\ref{i2}).}
Thus one is  led to study the following approximate system
 \begin{align}\label{eq:approx}
 \begin{aligned}
  \dif \varrho+\Div(\varrho\bu)\dif t&=\varepsilon\Delta\varrho\,\dif t,\\
  \dif(\varrho\bu)+\big[\Div(\varrho\bu\otimes\bu)
+a\nabla \varrho^\gamma+\delta\nabla\varrho^\Gamma \big]\dif t&= \ep \Delta (\vr \vu) + \Div\mathbb{ S}(\nabla_x \vu)\,\dif t+\mathbb{G}(\varrho,\varrho\bu) \,\dif W,
\end{aligned}
 \end{align}
where $\Gamma>\max\{\frac{9}{2},\gamma\}$. {For technical reasons, explained in detail in \cite{feireisl1}, 
the two limits $\ep \to 0$, $\delta\to 0$ must be distinguished and performed in this order.}   

{To find stationary solutions for \eqref{eq:approx} with $\varepsilon>0$ and $\delta>0$ fixed, two additional approximation layers are needed.} {Namely, a suitable  Faedo-Galerkin approximation of \eqref{eq:approx} of dimension $N\in\mn$, with certain truncations of various nonlinear terms (corresponding to a parameter $R \in\mn$).}
Letting $R\rightarrow\infty$ gives a unique solution to the Faedo-Galerkin approximation.
The passage to the limit as $N\rightarrow\infty$ yields existence of a solution to \eqref{eq:approx}.
Except for the first passage to the limit, we always employ the stochastic compactness method. However, due to the delicate structure of \eqref{i1}--\eqref{i3} it is necessary to work with weak topologies and therefore we are lead to Jakubowski's generalization of the classical Skorokhod representation theorem  \cite[Theorem 2]{jakubow}. It applies to a large class of topological spaces, the so-called quasi-Polish spaces, including (but not limited to) separable Banach spaces equipped with weak topologies.

Another important ingredient of the proof is then the identification of the limit in the nonlinear terms. To be more precise, two main difficulties arise. First, the passage to the limit in the terms that depend nonlinearly on $\vr$ (i.e. the pressure term and the stochastic integral) cannot be performed directly since strong convergence of the approximate densities does not follow from the compactness argument. This issue appears already in the deterministic setting and is overcome by a technique based on regularity of the effective viscous flux introduced by Lions \cite{LI4}. A suitable stochastic version of this method was developed in \cite{BrHo} to treat \eqref{i1}--\eqref{i3}. {Note, however, that the stationary problem is rather different from the initial--value problem, where
compactness of the initial density field can be incorporated by a suitable choice of the initial data. Here, in analogy with the deterministic 
approach developed in \cite{EF56}, compactness of the denisty must be recovered from stationarity of the flow.}

The second difficulty one has to face arises in the passage of the limit in the stochastic integral. Indeed,  one has to deal with a sequence of stochastic integrals driven by a sequence of Wiener processes. One possibility is to pass to the limit directly and such technical convergence results appeared in several works (see \cite{bensoussan} or \cite{krylov}), a detailed proof can be found in \cite{debussche1}. Another way is to show that the limit process is a martingale, identify its quadratic variation and apply an integral representation theorem for martingales, if available (see \cite{daprato}). The existence theory for \eqref{i1}--\eqref{i3} developed in \cite{BrHo} relies on neither of those and follows a rather new general and elementary method that was introduced in \cite{on1} and already generalized to different settings.

The main goal of the present paper is {to show the existence of stationary solutions to \eqref{i1}--\eqref{i3}
in the framework of weak martingale solutions introduced in \cite{BrHo}. Although the multi-level approximation procedure is 
identical with that used in \cite{BrHo}, the uniform estimates necessary for the existence theory are in general not suitable} to study the long-time behavior of the system. They are based on the application of Gronwall's lemma and therefore grow exponentially with the final time $T$. Hence, the major challenge is to derive new estimates which are uniform with respect to all the approximation parameters as well as in $T$. This is the heart of the paper. Let us point out that the standard methods used for the incompressible system, as for instance in \cite{FlaGat}, \cite{FlaRom}, are not applicable in the compressible case. Indeed system \eqref{i1}--\eqref{i3} is of mixed hyperbolic-parabolic type and the dissipation term does not contain  the density. Consequently, the  forcing terms on the right-hand side of the energy balance cannot be absorbed in the dissipative term appearing on the left-hand side {in an obvious straightforward manner.}

Furthermore, it does not seem to be possible to find {universal estimates that would be uniform} in all the parameters $R,N,\varepsilon,\delta$ as well as in $T$.
Instead, during each approximation step we develop new estimates which are then used for the particular passage to the limit at hand. More precisely, at the {starting} level, that is for fixed parameters $R,N\in\mn,$ $\varepsilon,\delta>0$, we show existence, uniqueness and continuous dependence on the initial condition. Thus, the resulting system is Markovian and the transition semigroup is Feller. Consequently, the existence of invariant measures can be shown with the help of the standard Krylov--Bogoliubov method in the infinite-dimensional setting. This generates a family of approximate stationary solutions. Note that we loose uniqueness already after the first passage to the limit (in $R$). Hence the usual Krylov--Bogoliubov approach cannot be employed anymore, and even the concept of invariant measure becomes ambiguous. To overcome this problem we construct stationary solutions  on the next level as limits of the corresponding approximate stationary solutions from the previous level.

At each approximation step, there are essentially three necessary estimates: for the energy, the velocity and the pressure. At the deepest level, we are able to obtain the first two estimates uniformly in $R,N$ but the third one depends on all the parameters $R,N,\varepsilon,\delta$ and is therefore not suitable for any limit procedure. The key observation is that these estimates may be significantly improved if we take stationarity into account. Therefore, working directly with stationary solutions given by the Krylov--Bogoliubov method, we derive an estimate for the energy as well as the velocity which is uniform in all the approximation parameters. The estimate for the pressure is more delicate and has to be reproved at each level by applying a suitable test function to \eqref{i1}--\eqref{i3}.
The proof is then concluded by performing the limit for vanishing approximation parameters based on a combination of deterministic and probabilistic tools, similarly to \cite{BrHo}.

{
It is remarkable that our result holds for the same range of the adiabatic exponent $\gamma > \frac{3}{2}$ as in the nowadays available existence theory. Note that the relevant deterministic problem, namely the existence of bounded absorbing sets and attractors require a rather inconvenient technical restriction
$\gamma > \frac{5}{3}$, see \cite{EF56}, \cite{FP14}. Indeed, consider the iconic example of the driving force $\vr \vc{f}(x) \D W$ in (\ref{i2}). 
If we replace it by the deterministic forcing $\vr \vc{f}(x) \D t$, then, to the best of our knowledge, it is not known if the global-in-time weak solutions 
remain uniformly bounded for $t \to \infty$ for $\gamma$ in the physically relevant range $1 \leq \gamma \leq 5/3$. On the other hand, the 
stochastic forcing $\vr \vc{f}(x) \D W$ gives rise to stationary solutions for any $\gamma > 3/2$ as shown in Theorem \ref{Tm1}. The reason is
the cancelations of certain terms in the energy balance due to stochastic averaging. We therefore observe a kind of regularizing effect 
due to the presence of noise.
Note, however, that the growth conditions imposed on the diffusion coefficients $\mathbb{G}(\vr, \vr \vu)$ appearing in the driving term
are more restrictive than in \cite{BrHo}. }

The rest of the paper is devoted to the proof of {existence of a stationary solution to the compressible Navier--Stokes system as stated in} Theorem \ref{Tm1}
{below}. The precise setting is given in Section \ref{m}. In Section \ref{b}, we introduce the basic finite-dimensional approximation and construct
a family of approximate solutions adapting the standard Krylov--Bogoliubov method. In Section \ref{d}, we develop global-in-time estimates for stationary solutions and pass to the limit $R\to\infty$ and $N\to\infty$. Section \ref{L} is devoted to the vanishing viscosity limit, i.e. $\varepsilon\to0$. Finally, in Section \ref{P}, we perform the limit for vanishing artificial pressure, i.e. $\delta\to0$, obtaining the desired stationary solution, the existence of which is claimed in Theorem \ref{Tm1}.

\section{Mathematical framework}
\label{m}

\subsection{Boundary conditions}

Although the boundary conditions in the real world applications may be quite complicated and of substantial influence on the fluid motion, our goal is to focus on the
effect of stochastic perturbations imposed through stochastic volume forces. Accordingly, we
consider the periodic boundary conditions, where the physical domain may be identified with the flat torus
\[
\mathbb{T}^3 \equiv \left( [-1,1]|_{\{ -1,1 \}} \right)^3.
\]
On the other hand, however, our method leans essentially on the dissipative effect of viscosity represented by $\mathbb{S}$ in (\ref{i2}).
In particular, it is convenient to keep a kind of Korn--Poincar\' e inequality in force. Following the idea of Ebin \cite{Eb}, we consider the 
{physically relevant} \emph{complete slip} conditions

\color{black}

\Cbox{Cgrey}{

\begin{equation} \label{i4}
\vu \cdot \vc{n}|_{\partial \mathcal O} = 0, \ [\mathbb{S}(\Grad \vu) \cdot \vc{n}] \times \vc{n}|_{\partial \mathcal O} = 0
\end{equation}
imposed on the boundary of the cube
\[
\mathcal{O} = [0,1 ]^3.
\]
\color{black}

}

The crucial observation is that
the constraint (\ref{i4}) is automatically satisfied by \emph{periodic} functions $\vr$, $\vu$ defined on torus $\mathbb{T}^3$ and belonging to the symmetry class
\begin{equation} \label{i6}
\begin{split}
\vr(t, -x) &= \vr(t,x) \qquad \ x \in \mathbb{T}^3,\\
u^i(t, \cdot, - x_i, \cdot) &= - u^{i}(t, \cdot, x_i, \cdot)\qquad i = 1, 2, 3 , \\
u^i(t, \cdot, - x_j, \cdot) &= u^i(t, \cdot, x_j, \cdot) \qquad   i \ne j, \ i,j=1, 2, 3,
\end{split}
\end{equation}
cf. \cite{Eb}. In such a way, we may eliminate the problems connected with the presence of physical boundary by considering
periodic functions defined on $\mathbb{T}^3$ and belonging, in addition,  to the symmetry class (\ref{i6}). Note that
for $\vu$ in the class (\ref{i6}), we have  Korn--Poincar\' e inequality
\begin{equation} \label{i7}
\intO{ \mathbb{S} (\Grad \vu) : \Grad \vu } \geq c_{KP} \| \vu \|^2_{W^{1,2}(\mathbb{T}^3; \R^3)}.
\end{equation}

\color{black}

In addition, we prescribe the total mass

\Cbox{Cgrey}{
\begin{equation} \label{i5}
\intO{ \vr(t,x) } = M_0,\qquad t\in[0,\infty),
\end{equation}
where  $M_0 > 0$ is {a deterministic constant.} 
}
The assumption that $M_0$ is deterministic is taken for simplicity, in order to avoid unnecessary technicalities. A more general case of random $M_0$ satisfying
\begin{equation}\label{i5aa}
\underline m\leq M_0 \leq \overline m \qquad\p\text{-a.s.}
\end{equation}
for some deterministic constants $\underline m,\overline m\in (0,\infty)$ can  also be considered. In that case, one would prescribe the law of $M_0$ such that \eqref{i5aa} holds.

\subsection{Stochastic setting}
\label{s}

{We consider a cylindrical $( \mathfrak{F}_t )_{t \geq 0}$-Wiener process defined on a stochastic basis
\[
\bas,
\]
with a probability space $(\Omega,\mathfrak{F},\prst)$, and a right-continuous complete filtration
$( \mathfrak{F}_t )_{t \geq 0}$. 
Formally, it is given by $W(t)=\sum_{k\geq1}e_k W_k(t) $ with $(W_k)_{k\in\mn}$ being mutually independent real-valued standard Wiener processes relative to $(\mf_t)_{t\geq0}$. Here $(e_k)_{k\in\mn}$ denotes a complete orthonormal system in a sepa\-rable Hilbert space $\mathfrak{U}$ (e.g. $\mathfrak{U}=L^2(\mt)$ would be a natural choice).}
The stochastic integral in \eqref{i2} is understood in the following sense
\[
\int \mathbb{G}(\varrho,\varrho\vu) \,\D  W =\sum_{k = 1}^\infty \int \mathbb{G}_k(x, \vr, \vr\vu)e_k  \,\D  W_k=: \sum_{k = 1}^\infty \int \vc{G}_k(x, \vr, \vr\vu)  \,\D  W_k,
\]
where the one-dimensional summands on the right-hand side are standard It\^o-type stochastic integrals. In agreement with (\ref{i6}), we suppose that the functions $\vc{G}_k = \vc{G}_k(x, \rho, \vc{q})$ satisfy
\begin{equation} \label{i8}
\begin{split}
G^i_k (\cdot, -x_i, \cdot, - q^i, \cdot ) &= - G^i_k (\cdot, x_i, \cdot, q^i ,\cdot), \qquad i=1,2, 3,\\
G^i_k (\cdot, -x_j, \cdot, - q^j , \cdot) &= G^i_k (\cdot, x_j, \cdot, q^j, \cdot), \qquad i \ne j, \ i,j=1, 2, 3.
\end{split}
\end{equation}

\begin{Remark} \label{RC1}

The meaning of (\ref{i8}) is to keep {the spatially periodic solutions} in the symmetry class (\ref{i6}) as long as the initial data belong to
(\ref{i6}) $\p$-a.s.

\end{Remark}

Finally, we define the auxiliary space $\mathfrak{U}_0\supset\mathfrak{U}$ via
$$\mathfrak{U}_0=\bigg\{v=\sum_{k\geq1}\alpha_k e_k;\;\sum_{k\geq1}\frac{\alpha_k^2}{k^2}<\infty\bigg\},$$
endowed with the norm
$$\|v\|^2_{\mathfrak{U}_0}=\sum_{k\geq1}\frac{\alpha_k^2}{k^2},\qquad v=\sum_{k\geq1}\alpha_k e_k.$$
Note that the embedding $\mathfrak{U}\hookrightarrow\mathfrak{U}_0$ is Hilbert-Schmidt. Moreover, trajectories of $W$ are $\prst$-a.s. in $C([0,T];\mathfrak{U}_0)$ (see \cite{daprato}). {For simplicity of the presentation, we often identify $\mathbb{G}(\vr,\vr\vu)$ as a Hilbert-Schmidt operator on $\mathfrak{U}$ with the sequence $\{\mathbf{G}_k(\vr,\vr\vu)\}_{k\in\mn}$ as an element of $\ell^2$.}

\color{black}

\subsection{Main result}

We use the concept of weak martingale solution introduced in  \cite{BrHo}.
In accordance with the available {\it a priori} bounds provided by the energy estimates, a suitable state space for
$[\vr, \vr \vu]$ is taken
\[
\vr \in L^\gamma(\mathbb{T}^3),\ \vr \vu \in L^{\frac{2 \gamma}{\gamma + 1}}(\mathbb{T}^3; \R^3),
\]
where $\gamma$ is the adiabatic exponent in the state equation (\ref{i3a}). Accordingly, we consider  initial laws
$\Lambda$ defined on  the Borel $\sigma$-algebra of the product space $L^\gamma (\mathbb{T}^3) \times L^{\frac{2 \gamma}{\gamma + 1}}(\mathbb{T}^3; \R^3)$.

\begin{Definition} \label{Dm1}
A quantity
\[
\left[ \bas ; \vr, \vu, W \right]
\]
is called a \emph{weak martingale solution} to problem (\ref{i1})--(\ref{i3}) {in $[0,T]$} with the initial law $\Lambda$ provided:
\begin{itemize}
\item
$\bas$ is a stochastic basis with a complete right-continuous filtration;
\item $W$ is an $( \mathfrak{F}_t )_{t \geq 0}$-cylindrical Wiener process;
\item the density $\vr$ satisfies $\vr \geq 0$, $t \mapsto \left< \vr(t), \psi \right> \in C([0,T])$ for any
$\psi \in C^\infty(\mathbb{T}^3)$
$\mathbb{P}$-a.s., the function $t \mapsto \left< \vr(t), \psi \right>$
{is $(\mathfrak{F}_t)$-adapted},
and
\begin{equation} \label{mom1}
\expe{ \sup_{t \in [0,T]} \| \vr(t) \|^n_{L^\gamma(\mathbb{T}^3)} } < \infty \ \mbox{{for a certain}}\ n > 1;
\end{equation}
\item {the velocity field $\vu \in L^2(\Omega \times (0,T); W^{1,2}(\mathbb{T}^3;\R^3))$ satisfies}
\begin{equation*} \label{mom2}
\expe{ \left( \int_0^T \| \vu(t) \|^2_{W^{1,2}(\mathbb{T}^3; \R^3)} \ \dt \right)^n } < \infty\ \mbox{\textcolor{black}{for a certain}}\ n > 1;
\end{equation*}
\item the momentum $\vr \vu$ satisfies $t \mapsto \left< \vr \vu(t), \phi \right> \in C([0,T])$ for any $\phi \in C^\infty(\mathbb{T}^3;\R^3)$
$\mathbb{P}$-a.s., the function $t \mapsto \left< \vr \vu(t), \phi \right>$ \textcolor{black}{is $(\mathfrak{F}_t)$-adapted},
\begin{equation} \label{mom3}
\expe{ \sup_{t \in [0,T]} \left\| \vr \vu(t) \right\|^n_{L^{\frac{2 \gamma}{\gamma + 1}}{(\mathbb{T}^3;\R^3)} }} < \infty\ \mbox{\textcolor{black}{for a certain}}\  n > 1;
\end{equation}
\item $\Lambda=\mathbb{P}\circ \left( \vr(0), \vr \vu (0) \right)^{-1} $,
\item $\mathbb{G}(\vr,\vr\vu)=\{\mathbf{G}_k(\vr,\vr\vu)\}_{k\in\mn}\in L^2(\Omega\times (0,T),\mathcal{P},\dif\p\otimes\dt;\ell^2(W^{-b,2}(\mt;\R^3)))$ for some $b>\frac{3}{2}$, where $\mathcal{P}$ denotes the progressively measurable $\s$-field associated to $(\mathfrak{F}_t)_{t\geq 0}$;
\item for all test functions $\psi \in C^\infty(\mathbb{T}^3)$, $\phi \in C^\infty(\mathbb{T}^3; \R^3)$ and all $t \in [0,T]$ it holds $\mathbb{P}$-a.s.
\begin{eqnarray}
\nonumber
\D \intO{ \vr \psi } &=& \intO{ \vr \vu \cdot \Grad \psi } \dt,
\\
\nonumber \D \intO{ \vr \vu \cdot \phi }  &=& \intO{ \Big[  \vr \vu \otimes \vu : \Grad \phi - \tn{S}(\Grad \vu) : \Grad \phi
+  p(\vr) \Div \phi \Big] } \dt \\ \nonumber
&&\quad+ \sum_{k=1}^\infty\intO{ \mathbf{G}_k (\vr, \vr \vu) \cdot \phi } \, \D W_k.
\end{eqnarray}

\end{itemize}

\end{Definition}

\begin{Remark} \label{mR1}

In addition to Definition \ref{Dm1}, we say that $[\vr, \vu]$ satisfy the complete slip boundary conditions (\ref{i4}), \textcolor{black}{if $[\vr(t, \cdot), \vr \vu(t, \cdot)]$} belong
to the symmetry class (\ref{i6}) for any $t \in [0,T]$ $\prst$-a.s.

\end{Remark}

\begin{Remark} \label{mR1bis}

\textcolor{black}{
Note that the statement about progressive measurability of the diffusion coefficients $\mathbb{G}(\vr, \vr \vu)$ is introduced for completeness, and,
as a matter of fact, can be deduced from the (weak) progressive measurability of $\vr$ and $\vr \vu$, see \cite{BrHo}.} 

\end{Remark}

\begin{Remark} \label{mR1bis+}

\textcolor{black}{
In contrast to the existence theory developed in \cite{BrHo}, the moments in (\ref{mom1})--(\ref{mom3}) are bounded up to a certain positive order $n$
rather then for \emph{all} $n > 1$ as in \cite{BrHo}. This is because the integrability of the moments for the initial--value problem is controlled by the initial  data.
}

\end{Remark}

\begin{Remark} \label{mR2}

Similarly to
\cite{BrFeHo2015A}, \textcolor{black}{we consider the class of \emph{dissipative} martingale solutions satisfying}, in addition to the stipulations specified in Definition \ref{Dm1},
an energy inequality. \textcolor{black}{Indeed some form of energy balance will be used
at every step of the construction of the stationary solution.} As a result, the stationary solution we obtain is also a dissipative martingale solution
in the sense of  \cite{BrFeHo2015A}. In addition, as in \cite{BrHo}, the equation of continuity \eqref{i1} is satisfied in the renormalized sense
\begin{eqnarray}
\D\intO{b(\vr)\psi} &=& \intO{ b(\vr)\vu \cdot\Grad\psi }\dt - \intO{ \big(b'(\vr)\vr-b(\vr)  \big)\Div\vu \,\psi }\dt
\end{eqnarray}
for every $\psi\in C^\infty(\mt)$, and every $b\in C^1([0,\infty))$ with $b'(z)=0$ for $z\geq M_b$ for some constant $M_b>0$. This is an essential tool to pass to the limit in the nonlinear pressure.

\end{Remark}

Due to the \textcolor{black}{specific} structure of the Navier-Stokes system \eqref{i1}--\eqref{i3}, 
a concept of stationarity must chosen accordingly. 
\textcolor{black}{We recall the standard definition of stationarity for \emph{continuous} processes ranging
in the Sobolev space $W^{k,p}$.}

\begin{Definition}\label{D2}
Let $k\in\mn_0$, $p\in[1,\infty)$ and let $\bfU=\{\bfU(t);t\in[0,\infty)\}$ be an $W^{k,p}(\mt)$-valued measurable stochastic process. We say that $\bfU$ is \emph{stationary} on $W^{k,p}(\mt)$ provided the joint laws
$$\mathcal{L}(\bfU(t_1+\tau),\dots, \bfU(t_n+\tau)),\quad \mathcal{L}(\bfU(t_1),\dots, \bfU(t_n))$$
on $[W^{k,p}(\mt)]^n$ coincide for all $\tau\geq0$, for all $t_1,\dots,t_n\in [0,\infty)$.

\end{Definition}

However, we observe that according to Definition \ref{Dm1}, the velocity $\vu$ is not a stochastic process in the classical sense. Indeed, its trajectories belong to $L^2(0,T;W^{1,2}(\mt;\R^3))$, i.e. are only defined almost everywhere in time. Therefore, even though the above definition of stationarity can be used for $[\vr,\vr\vu]$, it is not suitable to describe stationarity of $\vu$. To overcome this flaw, we  consider solutions as random variables ranging in the
space $L^q_{\rm loc}([0, \infty);W^{k,p}(\mt))$ as follows.

\begin{Definition}\label{D1}
Let $k\in\mn_0$, $p,q\in[1,\infty)$ and let $\bfU$ be an $L^q_{\rm loc}([0,\infty);W^{k,p}(\mt))$-valued random variable. Let $\mathcal S_\tau$ be the time shift on the space of trajectories given by $\mathcal{S}_\tau \bfU(t)=\bfU(t+\tau).$
We say that $\bfU$ is \emph{stationary} on $L^q_{\rm loc}([0,\infty);W^{k,p}(\mt))$ provided the laws
$\mathcal{L}(\mathcal{S}_\tau\bfU),$ $ \mathcal{L}(\bfU)$
on $L^q_{\rm loc}([0,\infty);W^{k,p}(\mt))$ coincide for all $\tau\geq0$.

\end{Definition} 

{Note that as Lemma \ref{lem:s} shows, it is actually sufficient to consider Definition \ref{D1} for $q=1$.}
As a matter of fact, the two concepts of stationarity {introduced 
in Definition \ref{D2} and Definition \ref{D1}} are equivalent as soon as the stochastic process in question is
continuous in time; or alternatively, if it is weakly continuous and satisfies a suitable uniform bound. Proofs of these statements are provided in Lemma \ref{l:equivD12} and Corollary \ref{l:equivD123} below. Furthermore, it can be shown that both notions of stationarity are stable under weak convergence, see Lemma \ref{lem:stac} and Lemma \ref{lem:stac2}.

Motivated by Definition \ref{D1}, we adapt the concept 
of stationarity introduced in the context of incompressible viscous fluids by Romito \cite{Romi}, cf. also the approach proposed by  
It\^{o} and Nisio \cite{ItNi}.

\begin{Definition}\label{Dm2}
A weak martingale solution $[\vr,\vu, W]$ to \eqref{i1}--\eqref{i3} is called {\em stationary} provided 
the joint law of the time shift 
$
\left[\mathcal{S}_\tau\vr, \mathcal{S}_\tau\vu, \mathcal{S}_\tau W - W \right]$ on
$$ L^1_{\rm loc}([0, \infty); L^\gamma(\mtor)) \times L^1_{\rm loc}([0, \infty); W^{1,2}(\mtor; \R^3)) 
\times C([0, \infty); \mathfrak{U}_0 )
$$
is independent of $\tau \geq 0$.
\end{Definition}

\begin{Remark}
{In accordance with the previous discussion, if $[\vr, \vu, W]$ is a stationary martingale solution of the Navier--Stokes system 
(\ref{i1})--(\ref{i3}) in the sense of Definition \ref{Dm2}, then the process $[\vr,\vr\bu]$ is stationary on $L^\gamma(\mt)\times L^\frac{2\gamma}{\gamma+1}(\mt;\R^3)$ in the sense of Definition \ref{D2};} whereas for $\vu$ we only have stationarity on $L^2_{\text{loc}}([0,\infty);W^{1,2}(\mt;\R^3))$ in the sense of Definition \ref{D1}.
\end{Remark}

The following theorem is the main result of the present paper. For notational simplicity, we restrict ourselves to the most difficult
and {physically relevant} case of three space dimensions.
However, our result extends to the two- and mono-dimensional case as well, even under the weaker assumption {$\gamma>1$ and 
$\gamma \geq 1$, respectively.}

\Cbox{Cgrey}{

\begin{Theorem} \label{Tm1}

Let $M_0 \in(0,\infty)$ be  given.
Let $p = p(\vr)$ be given by \eqref{i3a} with $\gamma > \frac{3}{2}$. Suppose that the
diffusion coefficients $\vc{G}_k$ belong to the symmetry class \eqref{i8} and there exist functions $\mathbf{g}_k\in C(\mathbb{T}^3 \times [0, \infty) \times \R^3; \R^3)$ and $\alpha_k\geq0$, $k\in\mn$, such that
\begin{equation} \label{m1}
\begin{split}
\vc{G}_k(x, \rho, \vc{q}) &= \rho \vc{g}_k (x, \rho, \vc{q}),\\
 \left|\nabla_{\rho, \vc{q}} \vc{g}_k (x, \rho, \vc{q}) \right| +
|\vc{g}_k (x, \rho, \vc{q})| &\leq \alpha_k, \qquad \sum_{k =1}^\infty \alpha^2_k  = G < \infty.
\end{split}
\end{equation}
Then problem \eqref{i1}--\eqref{i3}, \eqref{i4}, \eqref{i5} admits a {stationary martingale solution $[\vr, \vu, W]$.}

\end{Theorem}

}

Note that if for instance $\mathbf{G}_k(x,\rho,0)=0$ for all $x\in\mt$, $\rho\in [0,\infty)$ and $k\in\mn$, then \eqref{i1}--\eqref{i3} admits a trivial stationary solution, namely, $\vu\equiv 0$ and $\vr\equiv \text{const}.$ Nevertheless, Theorem \ref{Tm1} applies to more general diffusion coefficients $\mathbf{G}_k$ where such trivial solutions do not exist.

\begin{Remark}

Let us briefly discuss the noise term in the equation (\ref{i2}). Technically, $W$ is a cylindrical Wiener process. However, note that our approach covers also the standard case of distributed (space-dependent) noise under very natural conditions.
More specifically, consider the equation (\ref{i2}) written formally as

$$ \frac{ (\vr \vu)}{\dt} (t,x) + \Div (\vr \vu \otimes \vu) (t,x) + \Grad p(\vr) (t,x) = \Div \mathbb{S}(\Grad \vu) (t,x) + \sigma (x, \vr (t,x), \vr \vu (t,x)) \frac{\D W}{\dt}(t,x),$$

for $(t,x)\in (0,\infty) \times \mathbb{T}^3$, where the noise intensity $\sigma$ takes the form $\sigma (x,\rho, \vc{q})= \rho \sigma _1(x,\rho, \vc{q})$ and $\sigma _1 \in C^1(\mathbb{T}^3\times [0,\infty) \times \R^3; \R^3)$. Furthermore, $ \frac{\D W}{\dt}$ stands for a white in time, space-dependent noise, which is considered to be a formal derivative of an  $L^2 (\mathbb{T}^3; \R^3)$-valued Wiener process. Denote by $\Sigma$ the (trace class) incremental covariance of $W$; obviously there exists an orthonormal basis $(f_k)_{k\in\mn}$ in  $L^2 (\mathbb{T}^3; \R^3)$ and a sequence $(\lambda _k)_{k\in\mn},\ \lambda_k\ge 0,$ such that

$$ \Sigma f_k = \lambda _k f_k,\qquad W(t,x) = \sum_{k=1}^\infty \sqrt{\lambda _k}f_k(x)W_k(t),\qquad \sum_{k=1}^\infty \lambda _k <\infty,$$
where $(W_k)_{k\in\mn}$ is a sequence of independent, standard scalar Wiener processes.  In such case Definition~\ref{Dm1} yields a natural concept of rigorous solution to the system (\ref{i1})--(\ref{i3}) if we set
$$\vc{g}_k (x, \rho, \vc{q}) = \sigma _1(x, \rho, \vc{q}) f_k(x) \sqrt{\lambda _k},\qquad k\in \mathbb{N}.$$

If $\sigma_1$ is bounded, globally Lipschitz in $\rho, \ \vc{q}$, {uniformly in} $x$, $f_k \in C(\mt;\R^3)$ and $f_k$, are bounded on $\mt$, uniformly in $k\in \mathbb{N}$, then the noise term satisfies the condition (\ref{m1}), thereby  Theorem \ref{Tm1} is applicable to the present case.

\end{Remark}

{The rest of the paper is devoted to the proof of Theorem \ref{Tm1}.}

\section{Basic finite-dimensional approximation}
\label{b}

In this section, we introduce the zero-level approximate system to \eqref{i1}-\eqref{i3} and study its long-time behavior  for suitable initial data
{belonging to the symmetry class (\ref{i6}). More precisely, based on an energy estimate, Proposition \ref{prop:en}, and bounds for the density, Lemma \ref{lem:1608}, we apply the Krylov--Bogoliubov method to deduce the existence of an invariant measure.

{We point out that in accordance with hypothesis (\ref{i8}), the solutions
can be constructed to be
spatially periodic solutions, i.e. they belong to the symmetry class (\ref{i6}), as long as 
the initial data belong to the same class (\ref{i6}).} We always tacitly assume this fact  without specifying it explicitly in the future.
}

{Let}
\begin{align*}
\begin{aligned}
 H_N&=\bigg\{\bfw=[\bfw_1,\bfw_2,\bfw_3]:\ \bfw_i=\sum_{|{\bf m}|\leq N} a_{\bf m}[\bfw_i] \exp\left( {\rm i} {{\bf m}} \cdot x \right) ,\ 
|\vc{m}| \leq N  \bigg\}
\end{aligned}
\end{align*}
{be the space of trigonometric polynomials of order $N$, endowed with the Hilbert structure of the Lebesgue space $L^2(\mathbb{T}^3; \R^3)$, and let $\|\cdot\|_{H_N}$ denote the corresponding norm.}
Let
\[
\Pi_N : L^2(\mathbb{T}^3; \R^3) \to H_N
\]
be the associated $L^2$-orthogonal projection. Note that the following holds
\begin{align}\label{eq:project}
\|\Pi_N v\|_{L^p(\mathbb{T}^3;\R^3)}\leq\,c_p\|v\|_{L^p(\mathbb{T}^3;\R^3)}\qquad\forall v\in L^p(\mathbb{T}^3;\R^3),
\end{align}
and
$$\Pi_N v\to v\qquad\text{in}\qquad L^p(\mt;\R^3),$$
for any $p\in(1,\infty)$, cf. \cite[Chapter 3]{Gra}.

\subsection{Approximate field equations}

Fix $R\in\mn$, $N\in\mn$, $\ep > 0$, $\delta > 0$ and let {$\Gamma>\max\{\tfrac{9}{2},\gamma\}$}. The approximate solutions $\vr = \vr_N$, $\vu = \vu_N$, $\vu_N(t) \in H_N$ for any $t$, are constructed to satisfy the following system of equations
\begin{equation} \label{b1}
\begin{split}
\D \vr &+ \Div (\vr [\vu]_R ) \dt = \ep \Delta \vr \dt - 2 \ep \vr \dt + H \left( \frac{1}{M_0} \intO{ \vr } \right) \dt,\\
\D \intO{ \vr \vu \cdot \varphi} &- \intO{ \vr [\vu]_R \otimes \vu : \Grad \varphi } \dt
- \intO{ a \vr^\gamma H ( \| \vu \|_{H_N} - R ) \Div \varphi } \dt  \\
&=  - \intO{ \mathbb{S}(\Grad \vu) : \Grad \varphi} \dt  + \sum_{k=1}^\infty\intO{ \varrho\,\Pi_N\bfg_k(\varrho,\varrho\bfu) \cdot \varphi } \ \D W_k \\
&\hspace{-.5cm}+ \ep \intO{ \vr \vu \cdot \Delta \varphi } \dt - 2 \ep \intO{ \vr \vu \cdot \varphi } \ \dt + \delta \intO{ \vr^\Gamma
H(\| u \|_{H_N} - R )\Div \varphi } \dt,
\end{split}
\end{equation}
for any test function $\varphi \in H_N$, where
\[
[u]_R = H \left( \|\vu \|_{{H_N}} - R \right) \vu
\]
with
\[
H \in C^\infty(\R), \quad H = \left\{ \begin{array}{l} 1 \ \mbox{on}\ (-\infty, 0],\\ \mbox{a decreasing function on}\
( 0 , 1), \\ 0 \ \mbox{on}\ [1, \infty). \end{array} \right.
\]

Note that the basic approximate system \eqref{b1} is not the same as the one from \cite{BrHo}, cf. \eqref{eq:approx}. To be more precise, in order to obtain global-in-time estimates we are forced to include two more ``stabilizing'' terms in the continuity equation and to modify the momentum equation accordingly. Nevertheless, similarly to  \cite[Section 3]{BrHo}, it can be shown that problem
(\ref{b1}) admits a unique strong pathwise solution for any initial data $[\varrho_0,(\varrho\vu)_0]$ satifying, for some $\nu>0$,
\begin{equation} \label{b1a}
\begin{split}
&\vr_0 \in C^{2 + \nu}(\mathbb T^3), \ 0 < \underline{\vr} < \vr_0 < \Ov{\vr}, \
(\vr \vu)_0 \in {C^2(\mathbb T^3; \R^3)}\ \prst\mbox{-a.s.},\\
&\expe{ \left( \intO{ \left[ \frac{| (\vr \vu)_0 |^2}{\vr_0} + \frac{a}{\gamma- 1} \vr_0^\gamma + \frac{\delta}{\Gamma - 1} \vr_0^\Gamma \right] } \right)^n } \leq c(n)\ \ \text{for all } 1\leq n<\infty.
\end{split}
\end{equation}
where $\underline{\vr}$, $\Ov{\vr}$ are deterministic constants, and
where the associated initial value of $\vu$ is uniquely determined by
\begin{equation*}
\vu_0 \in H_N, \ \intO{ \vr_0 \vu_0 \cdot \varphi } = \intO{ (\vr \vu)_0 \cdot \varphi } \ \  \mbox{for all} \ \varphi \in H_N.
\end{equation*}

\subsection{Basic energy estimates}

The energy estimates established in \cite[Section 3]{BrHo} are not well-suited for the construction of stationary solutions. Indeed, the application of Gronwall's Lemma leads to an exponentially (in time) growing right hand side.
In this subsection we derive improved energy estimates which overcome this problem and hold true uniformly in $t$. However, it is important to note that at this stage of the proof, we are not able to obtain estimates independent of all the approximation parameters, namely, the following bounds blow up as $\varepsilon\to0$. The necessary uniform estimates for the passage to the limit in $\varepsilon$ will be derived directly for stationary solutions in Section \ref{d}.

\begin{Proposition}\label{prop:en}
Let $(\varrho,\bfu)$ be a solution to \eqref{b1} starting from
\begin{equation} \label{b13}
\vr_0 =1 ,\quad (\varrho\vu)_0=\vu_0=0.
\end{equation}
Then the following bounds hold true.
\begin{equation} \label{b14b}
\expe{ \left( \intO{ \left[ \frac{1}{2} \vr |\vu|^2 + \frac{a}{\gamma - 1} \vr^\gamma + \frac{\delta}{\Gamma - 1} \vr^{\Gamma} \right] (\tau, \cdot) }
\right)^n } \leq
c\left( n,\varepsilon, G \right),\ n\in\mn,
\end{equation}
\begin{equation} \label{b14}
\frac{1}{T} \expe{ \int_0^T \left( \| \vu \|^2_{W^{1,2}(\mathbb T^3; \R^3)} + \frac{2a\ep}{\gamma} |\Grad \vr^{\gamma/2} |^2_{L^2(\mathbb T^3; \R^3)} +
\frac{2\delta \ep}{\Gamma} |\Grad \vr^{\Gamma/2} |^2_{L^2(\mathbb T^3; \R^3)} \right) \dt } \leq  c\left( \ep, G \right).
\end{equation}
\end{Proposition}

\begin{proof}
Applying It\^{o}'s chain rule to \eqref{b1} we deduce the basic energy balance
\begin{equation} \label{b2}
\begin{split}
&\D \intO{ \left[ \frac{1}{2} \vr |\vu|^2 + \frac{a}{\gamma - 1} \vr^\gamma + \frac{\delta}{\Gamma - 1} \vr^{\Gamma} \right] }
+ 2 \ep \intO{ \left[ \frac{1}{2} \vr |\vu|^2 + \frac{a\gamma }{\gamma - 1} \vr^\gamma + \frac{\delta \Gamma }{\Gamma - 1} \vr^{\Gamma} \right] } \dt \\
&\quad+ \intO{ \mathbb{S}(\Grad \vu) : \Grad \vu } \dt
+ \ep \intO{ \vr |\Grad \vu|^2 } \dt + \ep \intO{  \left( a \gamma \vr^{\gamma - 2} + \delta \vr^{\Gamma - 2} \right) |\Grad \vr |^2 } \dt
\\
&\quad+
 \ep \intO{ \frac{1}{2} H\left( \frac{1}{M_0} \intO{ \vr } \right) |\vu|^2 }\dt
\\
& =  \sum_{k=1}^\infty\intO{ \varrho\,\Pi_N\bfg_k(\varrho,\varrho\vu) \cdot \vu } \, \D W_k + \frac{1}{2} \sum_{k=1}^\infty\intO{ \frac{1}{\vr} |\varrho \Pi_N\vc{g}_k(\varrho,\varrho\vu)|^2 } \dt
\\
&\quad+ H\left( \frac{1}{M_0} \intO{ \vr } \right) \intO{ \left( \frac{a \gamma}{\gamma - 1} \vr^{\gamma - 1} + \frac{\delta \Gamma}{\Gamma - 1} \vr^{\Gamma - 1}
 \right) }  \dt,
\end{split}
\end{equation}
we refer the reader to \cite[Proposition 3.1]{BrHo} for details. In view of hypothesis (\ref{m1}) and the continuity of $\Pi_N$ \eqref{eq:project}, we have
\begin{align}\label{b4}
\begin{aligned}
\sum_{k=1}^\infty\intO{ \frac{1}{\vr} |\varrho\Pi_N\vc{g}_k(\varrho,\varrho\vu)|^2 }&\leq c \|\varrho\|_{L^\gamma(\mt)}\sum_{k=1}^\infty\|\mathbf{g}_k(\varrho,\varrho\vu)\|^2_{L^{2\gamma'}(\mt;\R^3)}\\
& \leq c \|\varrho\|_{L^\gamma(\mt)}\sum_{k=1}^\infty\|\mathbf{g}_k(\varrho,\varrho\vu)\|^2_{L^\infty(\mt;\R^3)} \leq c(G)\|\varrho\|_{L^\gamma(\mt)},
\end{aligned}
\end{align}
where  $\tfrac{1}{\gamma}+\tfrac{1}{\gamma'}=1$.
Remark that the function $\hat{\vr} = \intO{ \vr }$ satisfies the (deterministic) ODE
\begin{equation} \label{b3}
\frac{\D}{\dt} \hat{\vr} = - 2 \ep \hat{\vr} + H \left( \frac{\hat{\vr}}{M_0}  \right).
\end{equation}
In particular, the function $\hat{\vr}$ is bounded by a constant depending solely on the initial mass $M_0$.
Taking expectation in (\ref{b2}) leads to
\begin{equation*}
\begin{split}
&\frac{\D}{\dt} \expe{ \intO{ \left[ \frac{1}{2} \vr |\vu|^2 + \frac{a}{\gamma - 1} \vr^\gamma + \frac{\delta}{\Gamma - 1} \vr^{\Gamma} \right] } }
+ 2 \ep \expe{ \intO{ \left[ \frac{1}{2} \vr |\vu|^2 + \frac{a \gamma}{\gamma-1}  \vr^\gamma  + \frac{\delta \Gamma}{\Gamma-1}  \vr^\Gamma  \right] } } \\
&\quad+
\expe{ \intO{ \mathbb{S}(\Grad \vu) : \Grad \vu }
 }  + \ep \expe{ \intO{  \left( a \gamma \vr^{\gamma - 2} + \delta \Gamma \vr^{\Gamma - 2} \right) |\Grad \vr |^2 } }
\\
& \leq c\left( G \right)\E\|\varrho\|_{L^\gamma(\mt)} + \expe{ H\left( \frac{1}{M_0} \intO{ \vr } \right) \intO{ \left( \frac{a \gamma}{\gamma - 1} \vr^{\gamma - 1} + \frac{\delta \Gamma}{\Gamma - 1} \vr^{\Gamma - 1}
 \right) } }.
\end{split}
\end{equation*}
Now, we observe that both  terms on the right hand side can be estimated by the weighted Young inequality and then absorbed in the second term on the left hand side. This readily implies \eqref{b14b} for $n=1$ with an $\varepsilon$-dependent constant on the right hand side that blows up as $\varepsilon\to0$.
In addition, keeping (\ref{b13}) in mind and applying the Korn--Poincar\' e inequality (\ref{i7}), we deduce the estimate for the ergodic averages \eqref{b14}.

As the next step, we apply the It\^o formula to (\ref{b2}) to obtain, for $n\in\mn$,
\begin{align} \label{b4a}
\begin{aligned}
&\D \left( \intO{ \left[ \frac{1}{2} \vr |\vu|^2 + \frac{a}{\gamma - 1} \vr^\gamma + \frac{\delta}{\Gamma - 1} \vr^{\Gamma} \right] } \right)^n
+ 2 \ep n \left( \intO{ \left[ \frac{1}{2} \vr |\vu|^2 + \frac{a\gamma }{\gamma - 1} \vr^\gamma + \frac{\delta \Gamma }{\Gamma - 1} \vr^{\Gamma} \right] } \right)^n \dt\\
&\quad+ n \left( \intO{ \left[ \frac{1}{2} \vr |\vu|^2 + \frac{a}{\gamma - 1} \vr^\gamma + \frac{\delta}{\Gamma - 1} \vr^{\Gamma} \right] } \right)^{n-1}
 \intO{ \mathbb{S}(\Grad \vu) : \Grad \vu } \dt \\
&\quad+ \ep n \left( \intO{ \left[ \frac{1}{2} \vr |\vu|^2 + \frac{a}{\gamma - 1} \vr^\gamma + \frac{\delta}{\Gamma - 1} \vr^{\Gamma} \right] } \right)^{n-1}\intO{ \vr |\Grad \vu|^2 } \dt \\
&\quad+ \ep n \left( \intO{ \left[ \frac{1}{2} \vr |\vu|^2 + \frac{a}{\gamma - 1} \vr^\gamma + \frac{\delta}{\Gamma - 1} \vr^{\Gamma} \right] } \right)^{n-1}\intO{  \left( a \gamma \vr^{\gamma - 2} +\delta \vr^{\Gamma - 2} \right) |\Grad \vr |^2 } \dt
\\
&\quad+
 \ep n \left( \intO{ \left[ \frac{1}{2} \vr |\vu|^2 + \frac{a}{\gamma - 1} \vr^\gamma + \frac{\delta}{\Gamma - 1} \vr^{\Gamma} \right] } \right)^{n-1}\intO{ \frac{1}{2} H\left( \frac{1}{M_0} \intO{ \vr } \right) |\vu|^2 }\dt\\
& = n \sum_{k=1}^\infty\left( \intO{ \left[ \frac{1}{2} \vr |\vu|^2 + \frac{a}{\gamma - 1} \vr^\gamma + \frac{\delta}{\Gamma - 1} \vr^{\Gamma} \right] } \right)^{n-1} \intO{ \varrho\Pi_N\bfg_k(\varrho,\vr\vu) \cdot \vu } \, \D W_k \\
&\quad+ \frac{n}{2}
\left( \intO{ \left[ \frac{1}{2} \vr |\vu|^2 + \frac{a}{\gamma - 1} \vr^\gamma + \frac{\delta}{\Gamma - 1} \vr^{\Gamma} \right] } \right)^{n-1}
\sum_{k=1}^\infty\intO{ \frac{1}{\vr} |\varrho\Pi_N\vc{g}_k(\varrho,\vr\vu)|^2 } \dt
\\
&\quad+n \left( \intO{ \left[ \frac{1}{2} \vr |\vu|^2 + \frac{a}{\gamma - 1} \vr^\gamma + \frac{\delta}{\Gamma - 1} \vr^{\Gamma} \right] } \right)^{n-1}
\times \\
&\quad\quad\times H\left( \frac{1}{M_0} \intO{ \vr } \right) \intO{ \left( \frac{a \gamma}{\gamma - 1} \vr^{\gamma - 1} + \frac{\delta \Gamma}{\Gamma - 1} \vr^{\Gamma - 1}
 \right) } \ \dt \\
&\quad+ \frac{n (n-1)}{2} \left( \intO{ \left[ \frac{1}{2} \vr |\vu|^2 + \frac{a}{\gamma - 1} \vr^\gamma + \frac{\delta}{\Gamma - 1} \vr^{\Gamma} \right] } \right)^{n-2}\sum_{k=1}^\infty\left( \intO{ \varrho\Pi_N\vc{g}_k(\vr,\vr\vu) \cdot \vu } \right)^2 \dt\\
&=:\mathcal K.
\end{aligned}
\end{align}
By virtue of (\ref{m1}) and the continuity of $\Pi_N$ \eqref{eq:project},
\begin{equation} \label{b4b}
\begin{split}
\sum_{k=1}^\infty\left( \intO{ \varrho\Pi_N\vc{g}_k(\vr,\vr\vu) \cdot \vu } \right)^2 &\leq \sum_{k=1}^\infty \left\| \sqrt{\varrho}\,\Pi_N\vc{g}_k(\vr,\vr\vu) \right\|^2_{L^2(\mathbb T^3;\R^3)} \| \sqrt{\vr} \vu \|^2_{L^2(\mathbb T^3; \R^3)} \\ &\leq
c\sum_{k=1}^\infty \|\varrho\|_{L^\gamma(\mathbb T^3)}\|\Pi_N\bfg_k(\vr,\vr\vu)\|^2_{L^{2\gamma'}(\mathbb T^3;\R^3)} \| \sqrt{\vr} \vu \|^2_{L^2(\mathbb T^3; \R^3)} \\
 &\leq
c\sum_{k=1}^\infty \|\varrho\|_{L^\gamma(\mt)}\|\bfg_k(\vr,\vr\vu)\|^2_{L^{2\gamma'}(\mathbb T^3;\R^3)} \| \sqrt{\vr} \vu \|^2_{L^2(\mathbb T^3; \R^3)} \\
 &\leq
c(G)\|\varrho\|_{L^\gamma(\mathbb T^3)} \| \sqrt{\vr} \vu \|^2_{L^2(\mathbb T^3; \R^3)} \\
&\leq
c(G)\|\vr\|_{L^\gamma(\mt)} \intO{ \left[ \frac{1}{2} \vr |\vu|^2 + \frac{a}{\gamma - 1} \vr^\gamma + \frac{\delta}{\Gamma - 1} \vr^{\Gamma} \right] }.
\end{split}
\end{equation}
Therefore, passing to expectations, the right hand side of \eqref{b4a} may be estimated by
\begin{align*}
\E\mathcal K&\leq
n \left( \intO{ \left[ \frac{1}{2} \vr |\vu|^2 + \frac{a}{\gamma - 1} \vr^\gamma + \frac{\delta}{\Gamma - 1} \vr^{\Gamma} \right] } \right)^{n-1}
\intO{ \left( \frac{a \gamma}{\gamma - 1} \vr^{\gamma - 1} + \frac{\delta \Gamma}{\Gamma - 1} \vr^{\Gamma - 1}
 \right) } \ \dt \\
&\quad+c\left( n,  G \right)\E\left( \intO{ \left[ \frac{1}{2} \vr |\vu|^2 + \frac{a}{\gamma - 1} \vr^\gamma + \frac{\delta}{\Gamma - 1} \vr^{\Gamma} \right] } \right)^{n-1}
\|\vr\|_{L^\gamma(\mt)}\dt .
\end{align*}
Now, after application of the weighted Young inequality, both these terms can be absorbed in the second term on the left hand side of \eqref{b4a}, yielding a constant that blows up as $\varepsilon\to0$.
Hence we may infer \eqref{b14b} for any solution of (\ref{b1}) starting from regular initial data (\ref{b1a}).
\end{proof}

\subsection{Regularity of the density}

Making use of the of the additional damping terms in the first equation in \eqref{b1}, we are able to show strong statements about the regularity of the solution depending on the parameters.

\begin{Lemma}\label{lem:1608}
Let $\bfu\in C([0,\infty);H_N)$.
Let $\varrho$ be a classical solution to
\begin{align}\label{eq:1508}
\partial_t\vr + \Div (\vr [\vu]_R ) = \ep \Delta \vr  - 2 \ep \vr  + H \left( \frac{1}{M_0} \intO{ \vr } \right)
\end{align}
with $\varrho(0)\in C^{2+\nu}(\mathbb T^3)$ such that $\varrho(0)>0$ and $\int_{\mathbb T^3} \varrho(0)\dx\leq \overline m$.\begin{itemize}
\item[(a)] Then we have
\begin{equation} \label{b23}
 \| \vr (\tau, \cdot) \|_{W^{k,p} (\mathbb T^3)}  \leq c(\overline m,k,p, N, R, \ep)\quad \forall \tau\geq1
\end{equation}
for all $k\in\mathbb N$ and $p<\infty$.
\item[(b)]
There exists a (deterministic) constant $\underline{\vr} = \underline{\vr}(\overline m, N, R, \ep) > 0$ such that
\begin{equation} \label{b21}
\vr(\tau,\cdot) \geq \underline{\vr} \quad \forall\tau\geq1.
\end{equation}
\end{itemize}
In particular, the constants are independent of $\bfu$.
\end{Lemma}

\begin{proof}
We start with equation (\ref{b3}) for the density averages that is independent of $\bfu$. Since (\ref{b3}) is a first order (deterministic) ODE an easy observation shows
\begin{align}\label{eq:1508b}
\hat{\varrho}(t)\rightarrow M_\varepsilon\quad \text{as}\quad t\rightarrow\infty,
\end{align}
where $M_\varepsilon>0$ is the unique solution to the equation $2\varepsilon M_\varepsilon=H\big(\frac{M_\varepsilon}{M_0}\big)$. The convergence above is uniform  in the sense that for every $\k>0$ there is $T=T(\overline{m},\varepsilon,\k)$ deterministic such that
$|\hat{\varrho}(t)-M_\varepsilon|<\k$ for all $t\geq T$.

The next step is to show that $\varrho$ is uniformly bounded from below as claimed in (b).
Returning to the equation of continuity, we have
\[
\partial_t \vr - \ep \Delta \vr + \Grad \vr [\vu ]_R = - (2 \ep + \Div [\vu]_R ) \vr + H\left( \frac{1}{M_0} \hat{\varrho} \right).
\]
Seeing that
\[
|\Div [\vu]_R | \leq D(R,N)
\]
for some constant $D(R,N)$,
we may use the comparison principle to deduce that
\[
\vr (t, \cdot) \geq \underline{\vr}(t),
\]
where $\underline{\vr}$ solves the equation
\begin{equation}\label{b111}
\frac{\D \underline{\vr}}{\dt} = - \underline{\vr} ( 2 \ep + D(R,N) ) + H\left( \frac{1}{M_0} \hat{\varrho}\right),
\quad 0 < \underline{\vr}(0)\leq \inf_{\mathbb T^3} \vr(0).
\end{equation}
In accordance with \eqref{eq:1508b} we have
\[
H\left( \frac{1}{M_0} \hat{\varrho}(t)\right) \to H\left( \frac{M_\varepsilon}{M_0} \right)=2\varepsilon M_\varepsilon>0 \quad \mbox{as}\quad t \to \infty.
\]
Since any solution to \eqref{b111} is asymptotically stabilized towards this equilibrium, we conclude that $\hat\varrho(t)>0$ for any $t>0$ and
\[
\underline{\vr} (t) \to \frac{ H\left( \frac{M_\varepsilon}{M_0} \right) }{2 \ep + D(R,N)} \quad \mbox{as}\quad t \to \infty
\]
and finally \eqref{b21} follows.

Now we are going to prove part (a). First, note that \eqref{eq:1508b}
implies
\begin{align}\label{eq:1508c}
\hat{\varrho}(t)=\|\varrho(t)\|_{L^1(\mathbb T^3)}\leq\,c(\overline{m}).
\end{align}
We apply maximal regularity theory (see e.g. \cite{HP}) to the equation \eqref{eq:1508} to obtain
\begin{align*}
\|&\partial_t\varrho\|_{L^2(T,T+1;W^{-2,q}(\mathbb T^3))}+\|\Delta\varrho\|_{L^2(T,T+1;W^{-2,q}(\mathbb T^3))}\\&\leq \,c\,\Big(\|\varrho(T)\|_{W^{-1,q}(\mathbb{T}^3)}+\|\Div(\varrho[\bfu]_R)\|_{L^2(T,T+1;W^{-2,q}(\mathbb T^3))}+\Big\|H \Big( \tfrac{1}{M_0} \intO{ \vr } \Big)\Big\|_{L^2(T,T+1;W^{-2,q}(\mathbb{T}^3))}\Big)
\end{align*}
where $q$ is chosen such $1<q<3/2$. Since $L^1(\mathbb T^3)\hookrightarrow W^{-1,q}$ and  we have \eqref{eq:1508c}
\begin{align*}
\|&\partial_t\varrho\|_{L^2(T,T+1;W^{-2,q}(\mathbb T^3))}+\|\varrho\|_{L^2(T,T+1;L^{q}(\mathbb{T}^3))}\\&\leq \,c\,\Big(\|\varrho(T)\|_{W^{-1,q}(\mathbb T^3)}+\|\varrho\|_{L^2(T,T+1;W^{-1,q}(\mathbb T^3))}+1\Big)\\
&\leq \,c\,\Big(\|\varrho(T)\|_{L^1(\mathbb T^3)}+\|\varrho\|_{L^2(T,T+1;L^1(\mathbb T^3))}+1\Big)\\
&\leq\,c\big(\|\varrho\|_{L^\infty(T,T+1;L^1(\mathbb T^3))}+1\big)\leq\,c,
\end{align*}
where $c$ depends on $R$ and $\varepsilon$ but is independent of $T$. Consequently, there is $\tau=\tau(T)\in[T,T+1]$ such that
$\varrho(\tau)$ is bounded in $L^q(\mathbb T^3)$ independently of $T$. A similar argument as above shows
\begin{align*}
\|&\partial_t\varrho\|_{L^2(\tau,\tau+1;W^{-1,q}(\mathbb T^3))}+\|\varrho\|_{L^2(T,T+1;W^{1,q}(\mathbb{T}^3))}\\&\leq \,c\,\Big(\|\varrho(\tau)\|_{L^{q}(\mathbb T^3)}+\|\varrho\|_{L^2(T,T+1;L^{q}(\mathbb T^3))}+1\Big)\leq \,c.
\end{align*}
So we have
\begin{align*}
\varrho\in L^2(T,T+1;W^{1,q}(\mathbb T^3))
\end{align*}
with a bound independent of $T$. Now, we can bootstrap the argument to obtain the claim.		
\end{proof}

\subsection{Approximate invariant measures}
\label{ssec:i}

With estimates (\ref{b14b}), (\ref{b14}), \eqref{b23} at hand, we are ready to apply the method of Krylov--Bogoliubov \cite[Section 3.1]{daPrZa} to construct an invariant measure for system (\ref{b1})
with fixed parameters $R$, $N$, $\ep$, and $\delta$. For $\underline{r}>0$ we define the set
$$\mathcal R=\mathcal R_{\underline{r}}=\{(r,\bfv)\in C^{2+\nu}(\mathbb T^3)\times H_N;\,\,\underline{r}^{-1}\leq r\leq \underline{r},\,\,\|\nabla r\|_{L^\infty(\mt)}\leq\underline{r}\}.$$
It will be the state space for solutions to \eqref{b1}. By $C_b(\mathcal R)$ we denote the space of continuous bounded functions on $\mathcal R$.

First of all, we remark that the approximate system \eqref{b1} can be solved using the Banach fixed point theorem as in \cite[Section 3]{BrHo}.
In  what follows, for an $\mathfrak{F}_s$-measurable $\mathcal{R}$-valued random variable $\eta$, we denote by $\bfU^\eta_{s,t}=(\varrho^\eta_{s,t},\vu^\eta_{s,t})$ the solution of \eqref{b1} at time $t$ starting at time $s$ from the initial condition $\eta$. If $s=0$ then we write simply $\bfU^\eta_{t}$. We obtain the following result.

\begin{Theorem}
There is $\underline{r}$ large enough such that the following holds. Let $0\leq s<t$ be given.
 Let $\eta$ be an $\mathfrak{F}_s$-measurable $\mathcal R$-valued initial condition. Then there exists $\bfU_s^\eta=(\varrho^\eta_{s},\bfu^\eta_s)\in L^2(\Omega;C([s,t];\mathcal R))$ which is the unique strong pathwise solution to \eqref{b1} starting from $\eta$ at time $s$. In addition, if $\eta_1$, $\eta_2$ are two such initial conditions then there is $\beta\in(0,2)$ such that
\begin{align}\label{eq:cont-dep}
\begin{aligned}
&\E\big\|\bfU^{\eta_1}_{s,t}-\bfU^{\eta_2}_{s,t}\big\|^2_{\mathcal R}\leq C(t-s,R,N,\varepsilon,\delta)\,\E\|\eta_1-\eta_2\|^\beta_{\mathcal R}.
\end{aligned}
\end{align}
\end{Theorem}

\begin{proof}
The existence of the unique strong pathwise solution was established in \cite[Section 3]{BrHo}. In addition, by means of Lemma \ref{lem:1608}, the solution belongs to $L^2(\Omega;C([s,t];\mathcal R))$ if we choose $\underline{r}$ large enough. Following \cite[Section 3]{BrHo} we obtain
\begin{align*}
\E\big\|\bfu^{\eta_1}_{s,t}-\bfu^{\eta_2}_{s,t}\big\|^2_{H_N}\leq \,\E\sup_{s\leq\sigma\leq t}\big\|\bfu^{\eta_1}_{s,\sigma}-\bfu^{\eta_2}_{s,\sigma}\big\|^2_{H_N}\leq C(t-s,R,N,\varepsilon,\delta)\,\E\|\eta_1-\eta_2\|^2_{\mathcal R}.
\end{align*}
Moreover, \cite[Lemma 2.2]{feireisl1} implies
\begin{align*}
&\sup_{s\leq\sigma\leq t}\big\|\varrho^{\eta_1}_{s,\sigma}-\varrho^{\eta_2}_{s,\sigma}\big\|_{W^{1,2}(\mathbb T^3)}\leq C(t-s,R,N,\varepsilon,\delta)\sup_{s\leq\sigma\leq t}\|\bfu^{\eta_1}_{s,\sigma}-\bfu^{\eta_2}_{s,\sigma}\|_{H_N}
\end{align*}
$\mathbb P$-a.s. and hence
\begin{align*}
\E\big\|\varrho^{\eta_1}_{s,t}-\varrho^{\eta_2}_{s,t}\big\|^\beta_{W^{1,2}(\mathbb T^3)}\leq C(t-s,R,N,\varepsilon,\delta)\,\E\|\eta_1-\eta_2\|^\beta_{\mathcal R}.
\end{align*}
for any $\beta>0$.
In order to obtain the final estimate we choose $l\in\mathbb N$ such that $W^{l,2}(\mathbb T^3)\hookrightarrow C^{2+\nu}(\mathbb T^3)$ and interpolate $W^{l,2}(\mathbb T^3)$ between $W^{l+1,2}(\mathbb T^3)$ and $W^{1,2}(\mathbb T^3)$. Using Lemma \ref{lem:1608} this implies
for some $\beta\in(0,2)$
\begin{align*}
\E\big\|\varrho^{\eta_1}_{s,t}-\varrho^{\eta_2}_{s,t}\big\|^2_{C^{2+\nu}(\mathbb T^3)}&\leq\,c\, \E\big\|\varrho^{\eta_1}_{s,t}-\varrho^{\eta_2}_{s,t}\big\|^2_{W^{l,2}(\mathbb T^3)}\\
&\leq\,c\, \E\big\|\varrho^{\eta_1}_{s,t}-\varrho^{\eta_2}_{s,t}\big\|^\beta_{W^{1,2}(\mathbb T^3)}\big\|\varrho^{\eta_1}_{s,t}-\varrho^{\eta_2}_{s,t}\big\|_{W^{l+1,2}(\mathbb T^3)}^{2-\beta}\\
&\leq C(t-s,R,N,\varepsilon,\delta)\,\E\|\eta_1-\eta_2\|^\beta_{\mathcal R}.
\end{align*}
\end{proof}

Let us now define the operators $P_t$ by
$$(P_t\varphi)(\eta):=\E\big[\varphi\big(\bfU^\eta_t\big)\big]\qquad \varphi\in C_b(\mathcal{R}).$$

\begin{Corollary}
  The equation \eqref{b1} defines a Feller Markov process, that is, $P_t:C_b(\mathcal{R})\to C_b(\mathcal{R})$ and
 \begin{equation}\label{eq:markov}
 \E [\varphi (\bfU^\eta_{t + s}) | \mathfrak{F}_t ] = (P_s
     \varphi) (\bfU^\eta_t) \qquad \forall \varphi \in C_b (\mathcal{R}), \hspace{1em} \forall
     \eta \in H, \hspace{1em} \forall t, s > 0.
 \end{equation}
  Besides, the semigroup property $P_{t+s}=P_t\circ P_s$ holds true.
\end{Corollary}

\begin{proof}
  The Feller property $P_t:C_b(\mathcal{R})\to C_b(\mathcal{R})$ is an immediate consequence of \eqref{eq:cont-dep} and the dominated convergence theorem.

  In order to establish the Markov property \eqref{eq:markov}, we shall prove that
  \[ \E [\varphi (\bfU^\eta_{t + s}) Z] =\E [(P_s \varphi) (\bfU^\eta_t)
     Z] \qquad \forall Z \in \mathfrak{F}_t . \]
  By uniqueness
  \[ \bfU^\eta_{t + s} = \bfU^{\bfU^\eta_t}_{t, t + s} \qquad \p\text{-a.s.}. \]
It is therefore sufficient to show that
  \[ \E [\varphi (\bfU^{\mathbf{V}}_{t, t + s}) Z] =\E [(P_s \varphi)
     (\mathbf{V}) Z] \]
holds true  for every $\mathfrak{F}_t$-measurable random variable $\mathbf{V}$. By
  approximation (one uses dominated convergence and the fact that $\bfV_n
  \rightarrow \bfV$ in $E$ implies $P_t \varphi (\bfV_n) \rightarrow P_t
  \varphi (\bfV)$ in $\R$ a.s.), it is enough to prove it for random
  variables $\bfV = \sum_{i = 1}^k \bfV^i \ind_{A^i}$ where $\bfV^i \in \mathcal{R}$ are deterministic and $(A^i)\subset \mathfrak{F}_t$ is a collection of disjoint sets such that $\cup_i A^i=\Omega$. Consequently, it is enough to prove it for every
  deterministic $\bfV\in E$. Now, the random variable $\bfU^{\bfV}_{t, t + s}$ depends
  only on the increments of the Brownian motion between $t$ and $t + s$, hence
  it is independent of $\mathfrak{F}_t .$ Therefore
  \[ \E [\varphi (\bfU^{\bfV}_{t, t + s}) Z] =\E [\varphi
     (\bfU^{\bfV}_{t, t + s})] \E [Z] . \]
  Since $\bfU^{\bfV}_{t, t + s}$ has the same law as $\bfU^{\bfV}_s$ by uniqueness,
  we have
  \[ \E [\varphi (\bfU^{\bfV}_{t, t + s}) Z] =\E [\varphi
     (\bfU^{\bfV}_s)] \E [Z] = P_s \varphi (\bfV) \E [Z]
     =\E [P_s \varphi (\bfV) Z] \]
  and the proof of \eqref{eq:markov} is complete.

  Taking expectation in \eqref{eq:markov} we get on the one hand
\[ \E [\E [\varphi (\bfU^\eta_{t + s}) | \mathfrak{F}_t
   ]] =\E [\varphi (\bfU^\eta_{t + s})] = (P_{t + s} \varphi) (\eta)
\]
and on the other hand
\[ \E [(P_s \varphi) (\bfU^\eta_t)] = (P_t (P_s \varphi)) (\eta) . \]
Thus the semigroup property follows.
\end{proof}

For an $\mathfrak{F}_0$-measurable random variable $\eta\in \mathcal{R}$, let $\mu_{t,\eta}$ denote the law of $\bfU_t^\eta$. If the law of $\eta$ is $\mu$ then it follows from the definition of the operator $P_t$ that $\mu_{t,\eta}=P_t^*\mu$. For the application of the Krylov--Bogoliubov method, we shall prove the following result.

\begin{Proposition}
Let the initial condition be given by \eqref{b13}, that is $\eta\equiv(1,0)\in\mathcal{R}$.
Then the set of laws
$$\left\{\frac{1}{T}\int_0^{T}\mu_{s,\eta}\,\dif s;\;T>0\right\}$$
is tight on $\mathcal{R}$.
\end{Proposition}

\begin{proof}
Recall that $\mu_{s,\eta}$ are laws on the space $\mathcal{R}$. In particular, the second component is finite dimensional whereas the first one  not.
Let $\mu_{s,\eta}^\varrho$ and $\mu_{s,\eta}^\bfu$ denote the marginals of $\mu_{s,\eta}$ corresponding respectively to the first and second component of the solution. That is, $\mu_{s,\eta}^\varrho$ is the law of $\varrho^\eta_s$ on $C^{2+\nu}(\mt)$ and $\mu_{s,\eta}^\bfu$ is the law of $\bfu^\eta_s$ on $H_N$. It is then enough to establish tightness of both following sets separately:
\begin{equation}\label{tight}
\left\{\frac{1}{T}\int_0^{T}\mu^\bfu_{s,\eta}\,\dif s;\;T>0\right\},\qquad \left\{\frac{1}{T}\int_0^{T}\mu^\varrho_{s,\eta}\,\dif s;\;T>0\right\}.
\end{equation}

 As a consequence of \eqref{b14} and the equivalence of norms on $H_N$ we have
\begin{align*}
\frac{1}{T}\E\bigg[\int_0^T\|\bfu^\eta_t\|_{H_N}^2\dt\bigg]\leq c(N,\varepsilon,  G).
\end{align*}
Consequently,  for compact sets
$$B_R:=\left\{\bfv\in H_N;\, \|\bfv\|_{H_N}\leq R\right\}\subset H_N$$
by means of Chebyshev inequality we obtain
\begin{align*}
\frac{1}{T}\int_0^T\mu_{s,\eta}^\bfu(B^c_R)\,\dif s&=\frac{1}{T}\int_0^T\p(\|\bfu^\eta_s\|_{H_N}>R)\,\dif s\leq \frac{1}{R^2}\frac{1}{T}\E\bigg[\int_0^T\|\bfu^\eta_t\|_{H_N}^2\dt\bigg],
\end{align*}
which in turn implies the tightness of the first set in \eqref{tight}.
In order to establish tightness in the second component, we define
$$B_R:=\left\{r\in W^{k,p}(\mt);\,\|r\|_{W^{k,p}(\mt)}\leq R\right\}.$$
For $p\in(1,\infty)$ and $k\in\mn$ sufficiently large  this is a compact set in $C^{2+\nu}(\mt)$ hence making use of \eqref{b23} we have
\begin{align*}
\frac{1}{T}\int_0^T\mu_{s,\eta}^\varrho(B^c_R)\,\dif s&=\frac{1}{T}\int_0^T\p(\|\varrho^\eta_s\|_{W^{k,p}}>R)\,\dif s\leq \frac{1}{R}\sup_{t\geq 0}\E\|\varrho^\eta_t\|_{W^{k,p}}
\end{align*}
and the desired tightness follows.
\end{proof}

Finally, the Krylov--Bogoliubov theorem \cite[Theorem 3.1.1]{daPrZa} applies and yields the following.

\begin{Corollary}\label{cor:inv}
Fix $R,\,N\in\mn$, $\varepsilon,\,\delta>0$. Then there exists an invariant measure $\mathcal{L}_{\varrho,\vu}$ for the dynamics given by \eqref{b1}. In addition, as consequence of \eqref{eq:1508b},
$$\mathcal{L}_{\varrho,\vu}[r\geq\underline\varrho]=1,\qquad \mathcal{L}_{\varrho,\vu}\bigg[\int_{\mathbb T^3}r\dx=M_\varepsilon\bigg]=1.$$
\end{Corollary}

\section{First limit procedures: $R\to\infty$, $N\to\infty$}
\label{d}

The existence of an invariant measure for the zero level approximate problem  \eqref{b1} implies the existence of a stationary solution $[\vr_R,\bfu_R]$. Our ultimate goal is  to perform successively the limits for $R \to \infty$, $N \to \infty$,
$\ep \to 0$, and finally $\delta \to 0$. Even though this may look like a straightforward modification of the arguments used in the existence proof
\cite{BrHo}, there are several new aspects that must be handled. First of all, the uniform bounds used in the existence proof \cite{BrHo} are controlled by the initial data. This is not the case for the stationary solution for which the
``initial value'' is not {\it a priori} known and the necessary estimates must be deduced from the energy balance (\ref{b4a}) using the fact that the solution
possesses the same law at any time. Moreover, the estimates derived in the previous section, that is, Proposition \ref{prop:en} and Lemma \ref{lem:1608} do not hold independently of the approximation parameters $R,N,\varepsilon,\delta$ and are therefore not suitable for the limit procedure.
In addition, since the point-values of the density are not compact, the proof of the strong convergence of the approximate densities based on continuity of the effective viscous flux must be modified.

Let $[\vr_{R}, \vu_{R}]$  be a solution of the approximate problem (\ref{b1}) whose law at (every) time $t$ is given by the invariant measure
$\mathcal{L}_{\vr_{R}, \vu_{R}}$ constructed in Corollary \ref{cor:inv}. As the first step, we show a new uniform bound for $[\vr_{R}, \vu_{R}]$ that can be deduce from the energy balance (\ref{b4a}). Note that at this stage, the estimate still blows up as $\varepsilon\to0$.

\begin{Proposition}\label{prop:90}
Let $[\vr_{R}, \vu_{R}]$ be a stationary solution to \eqref{b1} given by the invariant measure  from Corollary \ref{cor:inv}. Then we have for all $n\in\mn$ and  a.e. $t\in(0,\infty)$
\begin{equation} \label{d51x}
\begin{split}
& \expe{ \left( \intO{ \left[ \frac{1}{2} \vr_{R} |\vu_{R}|^2 + \frac{a\gamma }{\gamma - 1} \vr_{R}^\gamma+ \frac{\delta \Gamma }{\Gamma - 1} \vr_{R}^{\Gamma} \right] } \right)^n } \leq c(n, G,\varepsilon),
\end{split}
\end{equation}
$$
\expe{\|\vu_R\|_{W^{1,2}(\mt)}^2}\leq c(G,\ve).
$$
\end{Proposition}

\begin{proof}
After taking expectations in (\ref{b4a}), we observe that due to stationarity of $[\vr_R,\vu_R]$, the first term is constant in time, thus its time derivative vanishes. This is a consequence of Corollary \ref{cor:00}. By the same reasoning we may ultimately omit the time integrals in all the remaining expressions. Then we apply (\ref{b4}), (\ref{b4b}) to estimate the terms coming from the stochastic integral and  obtain
\begin{equation*}
\begin{split}
&\ep  \expe{ \left( \intO{ \left[ \frac{1}{2} \vr_{R} |\vu_{R}|^2 + \frac{a\gamma }{\gamma - 1} \vr_{R}^\gamma + \frac{\delta \Gamma }{\Gamma - 1} \vr_{R}^{\Gamma} \right] } \right)^n } \\
&\quad+  \expe{ \left( \intO{ \left[ \frac{1}{2} \vr_{R} |\vu_{R}|^2 + \frac{a}{\gamma - 1} \vr_{R}^\gamma + \frac{\delta}{\Gamma - 1} \vr_{R}^{\Gamma} \right] } \right)^{n-1}
 \intO{ \mathbb{S}(\Grad \vu_{R}) : \Grad \vu_{R} } } \\
&\quad+ \ep \expe{ \left( \intO{ \left[ \frac{1}{2} \vr_{R} |\vu_{R}|^2 + \frac{a}{\gamma - 1} \vr_{R}^\gamma + \frac{\delta}{\Gamma - 1} \vr_{R}^{\Gamma} \right] } \right)^{n-1}\intO{  \left( a \gamma \vr_{R}^{\gamma - 2} +\delta \vr_{R}^{\Gamma - 2} \right) |\Grad \vr_{R} |^2 } }
\\
&\leq c(n,G) \left( \intO{ \left[ \frac{1}{2} \vr_R |\vu_R|^2 + \frac{a}{\gamma - 1} \vr_R^\gamma + \frac{\delta}{\Gamma - 1} \vr_R^{\Gamma} \right] } \right)^{n-1}
\intO{ \left( \frac{a \gamma}{\gamma - 1} \vr_R^{\gamma - 1} + \frac{\delta \Gamma}{\Gamma - 1} \vr_R^{\Gamma - 1}
 \right) }  \\
&\quad+c\left( n,  G \right)\E\left( \intO{ \left[ \frac{1}{2} \vr_R |\vu_R|^2 + \frac{a}{\gamma - 1} \vr_R^\gamma + \frac{\delta}{\Gamma - 1} \vr_R^{\Gamma} \right] } \right)^{n-1}
\|\vr_R\|_{L^\gamma(\mt)}.
\end{split}
\end{equation*}
The application of the weighted Young inequality allows to absorb both terms on the right hand side into the first term on the left hand side. The claim follows.
\end{proof}

\begin{Proposition}\label{prop:1608a}
Let $[\vr_{R}, \vu_{R}]$ be a stationary solution to \eqref{b1} given by the invariant measure  from Corollary \ref{cor:inv}. Then we have for all $n\in\mn$, a.e. $T\in(0,\infty)$ and  $\tau>0$
\begin{equation} \label{b2xy}
\begin{split}
&\expe{\bigg(\sup_{t\in[T,T+\tau]}\intO{ \left[ \frac{1}{2} \vr_{R} |\vu_{R}|^2 + \frac{a}{\gamma - 1} \vr_{R}^\gamma + \frac{\delta}{\Gamma - 1} \vr_{R}^{\Gamma} \right] }\bigg)^n}\\
&\quad
+ 2 \ep\,\expe{ \bigg(\int_T^{T+\tau}\intO{ \left[ \frac{1}{2} \vr_{R} |\vu_{R}|^2 + \frac{a\gamma }{\gamma - 1} \vr_{R}^\gamma + \frac{\delta \Gamma }{\Gamma - 1} \vr_{R}^{\Gamma} \right] } \dt \bigg)^n}\\
&\quad+ \expe{\bigg(\int_T^{T+\tau} \|  \vu_{R}\|^2_{W^{1,2}(\mt;\R^3)}  \dt\bigg)^n}
+ \ep \,\expe{\bigg(\int_T^{T+\tau}\intO{  \left( a \gamma \vr_{R}^{\gamma - 2} + \delta \vr_{R}^{\Gamma - 2} \right) |\Grad \vr_{R} |^2 } \dt\bigg)^n}
\\
& \leq c(n,G,\ve,\tau),
\end{split}
\end{equation}
where the constant on the right hand side does not depend on $T$.
\end{Proposition}

\begin{proof}
Taking the $n$-th power and expectation in the energy balance \eqref{b2} and applying \eqref{b4}, \eqref{b4b} and the Korn--Poincar\'e inequality, we deduce
\begin{equation} \label{b2x}
\begin{split}
&\expe{\sup_{t\in[T,T+\tau]}\intO{ \left[ \frac{1}{2} \vr_{R} |\vu_{R}|^2 + \frac{a}{\gamma - 1} \vr_{R}^\gamma + \frac{\delta}{\Gamma - 1} \vr_{R}^{\Gamma} \right] }}^n\\
&\quad
+ 2 \ep\,\expe{ \int_T^{T+\tau}\intO{ \left[ \frac{1}{2} \vr_{R} |\vu_{R}|^2 + \frac{a\gamma }{\gamma - 1} \vr_{R}^\gamma + \frac{\delta \Gamma }{\Gamma - 1} \vr_{R}^{\Gamma} \right] } \dt }^n\\
&\quad+ \expe{\int_T^{T+\tau} \|  \vu_{R}\|_{W^{1,2}(\mt;\R^3)}  \dt}^n
+ \ep \,\expe{\int_T^{T+\tau}\intO{  \left( a \gamma \vr_{R}^{\gamma - 2} + \delta \vr_{R}^{\Gamma - 2} \right) |\Grad \vr_{R} |^2 } \dt}^n
\\
& \leq c(n)\, \expe{\intO{ \left[ \frac{1}{2} \vr_{R}(T) |\vu_{R}(T)|^2 + \frac{a}{\gamma - 1} \vr_{R}^\gamma(T) + \frac{\delta}{\Gamma - 1} \vr_{R}^{\Gamma}(T) \right] } }^n\\
&\quad+ c(n,G)\, \expe{\int_T^{T+\tau}\|\vr_{R}\|_{L^\g(\mt)}\int_{\mt}\vr_{R}|\vu_{R}|^2\,\dif x\,\dif t}^\frac{n}{2}+c(n,G)\,\expe{\int_T^{T+\tau}\|\vr_{R}\|_{L^\g(\mt)}\dt}^n
\\
&\quad+ c(n)\,\expe{\int_T^{T+\tau} \intO{ \left( \frac{a \gamma}{\gamma - 1} \vr_{R}^{\gamma - 1} + \frac{\delta \Gamma}{\Gamma - 1} \vr_{R}^{\Gamma - 1}
 \right) }  \dt}^n.
\end{split}
\end{equation}
The first term on the right hand side can be estimated due to \eqref{d51x} by a constant $c(n,G,\ve)$. The third term on the right hand side can be estimated by Young's inequality as follows
\begin{align}\label{b2xx}
\begin{aligned}
&\expe{\int_T^{T+\tau}\|\vr_{R}\|_{L^\g(\mt)}\dt}^n\\
&\qquad\leq \frac{\ve}{2}\,\expe{\int_T^{T+\tau}\intO{ \left[ \frac{1}{2} \vr_{R} |\vu_{R}|^2 + \frac{a\gamma }{\gamma - 1} \vr_{R}^\gamma + \frac{\delta \Gamma }{\Gamma - 1} \vr_{R}^{\Gamma} \right] } \dt}^n +c(n,\ve,\tau)
\end{aligned}
\end{align}
and then absorbed into the second term on the left hand side of \eqref{b2x}. A similar approach applies to the last term on the right hand side of \eqref{b2x}. For the remaining term we write
\begin{align*}
&\expe{\int_T^{T+\tau}\|\vr\|_{L^\g(\mt)}\int_{\mt}\vr_{R}|\vu_{R}|^2\,\dif x\,\dif t}^\frac{n}{2}
\leq \expe{\sup_{t\in[T,T+\tau]}\intO{\frac{1}{2}\vr_{R}|\vu_{R}|^2 } \int_T^{T+\tau} \|\vr_{R}\|_{L^\g(\mt)}\,\dif t }^\frac{n}{2}\\
&\leq \kappa\,\expe {\sup_{t\in[T,T+\tau]}\intO{\frac{1}{2}\vr_R|\vu_R|^2}}^n + c(\kappa)\,\expe{ \int_{T}^{T+\tau} \|\vr_R\|_{L^\g(\mt)}\dt}^n,
\end{align*}
where the last term can be again estimated as in \eqref{b2xx}. Choosing $\k$ sufficiently small yields the claim.
\end{proof}

In view of the uniform bounds provided by Proposition \ref{prop:1608a} , for fixed $\ep,\delta > 0$, the asymptotic limits for $R \to \infty$ and $N \to \infty$ can be carried over exactly as for the initial value problem in \cite[Section 3, Section 4]{BrHo}.
In the limit, we obtain the following approximate system.

\begin{itemize}
\item {\bf Regularized equation of continuity.}
\begin{equation} \label{L1}
\begin{split}
\int_{0}^{\infty} &\intO{ \left[ \vr \partial_t \varphi + \vr \vu \cdot \Grad \varphi \right] } \dt \\ &= \ep \int_0^\infty \intO{ \left[\Grad \vr \cdot \Grad \varphi
- 2 \vr \varphi \right] }\dt - 2\ep \int_0^\infty \intO{ M_{\ep} \varphi } \dt
\end{split}
\end{equation}
for any $\varphi \in \DC((0, \infty) \times \mt)$ $\prst$-a.s.
\item {\bf Regularized momentum equation.}
\begin{equation} \label{L2}
\begin{split}
\int_0^\infty \partial_t \psi\intO{  \vr \vu \cdot \varphi} \dt  &+ \int_0^\infty\psi  \intO{ \vr \vu \otimes \vu : \Grad \varphi } \dt
+ \int_0^\infty \psi \intO{ (a \vr^\gamma+\delta\varrho^\beta)  \Div \varphi } \dt  \\
&-   \int_0^\infty\psi  \intO{ \mathbb{S}(\Grad \vu) : \Grad \varphi} \dt  - \ep \int_0^\infty \psi \intO{  \vr \vu \cdot \Delta \varphi } \dt \\&- 2 \ep \int_0^\infty \psi \intO{ \vr \vu \cdot \varphi } \ \dt
=- \sum_{k=1}^{\infty}\int_0^\infty \psi \intO{ \mathbf{G}_k \cdot \varphi } \ \D W_k
\end{split}
\end{equation}
for any $\psi \in \DC((0, \infty))$, $\varphi \in C^\infty(\mt; \R^3)$ $\prst$-a.s.
\end{itemize}

To summarize, we deduce the following.

\begin{Proposition}\label{prop:1608b}
Let $\varepsilon,\delta>0$ be given. Then there exists a stationary weak martingale solution $[\varrho_\varepsilon,\bfu_\varepsilon]$ to \eqref{L1}--\eqref{L2}.
In addition, for $n\in\mn$ and  every  $\psi\in C^\infty_c((0,\infty))$, $\psi\geq 0$, the following generalized energy inequality holds true
\begin{align} \label{b2a}
\begin{aligned}
&-\int_0^\infty\partial_t\psi \left( \intO{ \left[ \frac{1}{2} \vr_\ve |\vu_\ve|^2 + \frac{a}{\gamma - 1} \vr_\ve^\gamma + \frac{\delta}{\Gamma - 1} \vr_\ve^{\Gamma} \right] } \right)^n\dt\\
&\quad+ 2 \ep n \int_0^\infty\psi\left( \intO{ \left[ \frac{1}{2} \vr_\ve |\vu_\ve|^2 + \frac{a\gamma }{\gamma - 1} \vr_\ve^\gamma + \frac{\delta \Gamma }{\Gamma - 1} \vr_\ve^{\Gamma} \right] } \right)^n \dt\\
&\quad+ n \int_0^\infty\psi\left( \intO{ \left[ \frac{1}{2} \vr_\ve |\vu_\ve|^2 + \frac{a}{\gamma - 1} \vr_\ve^\gamma + \frac{\delta}{\Gamma - 1} \vr_\ve^{\Gamma} \right] } \right)^{n-1}
 \intO{ \mathbb{S}(\Grad \vu_\ve) : \Grad \vu_\ve } \dt \\
& \leq n\sum_{k=1}^\infty \int_0^\infty\psi\left( \intO{ \left[ \frac{1}{2} \vr_\ve |\vu_\ve|^2 + \frac{a}{\gamma - 1} \vr_\ve^\gamma + \frac{\delta}{\Gamma - 1} \vr_\ve^{\Gamma} \right] } \right)^{n-1} \intO{ \varrho_\ve\,\bfg_k(\varrho_\ve,\vr_\ve\vu_\ve) \cdot \vu_\ve } \, \D W_k \\
&\quad+ \frac{n}{2} \int_0^\infty\psi
\left( \intO{ \left[ \frac{1}{2} \vr_\ve |\vu_\ve|^2 + \frac{a}{\gamma - 1} \vr_\ve^\gamma + \frac{\delta}{\Gamma - 1} \vr_\ve^{\Gamma} \right] } \right)^{n-1}
\sum_{k=1}^\infty\intO{ \frac{1}{\vr_\ve} |\varrho_\ve\,\vc{g}_k(\varrho_\ve,\vr_\ve\vu_\ve)|^2 } \dt
\\
&\quad+n \int_0^\infty\psi\left( \intO{ \left[ \frac{1}{2} \vr_\ve |\vu_\ve|^2 + \frac{a}{\gamma - 1} \vr_\ve^\gamma + \frac{\delta}{\Gamma - 1} \vr_\ve^{\Gamma} \right] } \right)^{n-1}
\times \\
&\qquad\qquad\times H\left( \frac{M_\ve}{M_0} \right) \intO{ \left[ \frac{a \gamma}{\gamma - 1} \vr_\ve^{\gamma - 1} + \frac{\delta \Gamma}{\Gamma - 1} \vr_\ve^{\Gamma - 1}
 \right] } \ \dt \\
&\quad+ \frac{n (n-1)}{2} \int_0^\infty\psi\left( \intO{ \left[ \frac{1}{2} \vr_\ve |\vu_\ve|^2 + \frac{a}{\gamma - 1} \vr_\ve^\gamma + \frac{\delta}{\Gamma - 1} \vr_\ve^{\Gamma} \right] } \right)^{n-2}\times\\
&\qquad\qquad\times\sum_{k=1}^\infty\left( \intO{ \varrho_\ve\,\vc{g}_k(\vr_\ve,\vr_\ve\vu_\ve) \cdot \vu_\ve } \right)^2 \dt.
\end{aligned}
\end{align}
\end{Proposition}

\begin{proof}
The proof follows the lines of \cite[Section 3, Section 4]{BrHo}. The first passage to the limit as $R\to\infty$ relies on a stopping time argument from \cite[Subsection 3.1]{BrHo} whereas the limit $N\to\infty$ employs the stochastic compactness method based on the Jakubowski-Skorokhod representation theorem \cite[Theorem 2]{jakubow} as in \cite[Section 4]{BrHo}. We point out that all the necessary estimates in \cite[Section 3, Section 4]{BrHo} come from the energy balance, which is controlled by the initial condition. In the present construction, the bound for the initial energy is replaced by the estimate \eqref{d51x} which holds true due to stationarity. Apart from that, the only difference to \cite{BrHo} is that we have to deal with path spaces containing an unbounded time interval, that is
\begin{align*}
L^q_{\text{loc}}([0,\infty);X),\quad (L^q_{\text{loc}}([0,\infty);X),w),\quad C_{\text{loc}}([0,\infty);(X,w)),
\end{align*}
where $q\in(1,\infty)$ and $X$ is a reflexive separable Banach space. Recall that $L^q_{\text{loc}}([0,\infty);X)$ is a separable metric space given by
\begin{align*}
(f,g)\mapsto \sum_{M\in\mn}2^{-M}\big(\|f-g\|_{L^q(0,M;X)}\wedge 1\big),
\end{align*}
and a set $\mathcal{K}\subset L^q_{\text{loc}}([0,\infty);X)$ is compact provided  the set
$$\mathcal{K}_M:=\{f|_{[0,M]};\,f\in\mathcal{K}\}\subset L^q(0,M;X)$$
is compact for every $M\in\mn$. On the other hand, the remaining two spaces are  (generally) nonmetrizable locally convex topological vector spaces, generated by the seminorms
$$
f\mapsto \int_0^M \langle f(t), g(t) \rangle_X\dt,\qquad M\in\mn,\,g\in L^{q'}(0,\infty; X^*),\,\tfrac{1}{q}+\tfrac{1}{q'}=1,
$$
and
$$
f\mapsto \sup_{t\in[0,M]}\langle f(t), g \rangle_X,\qquad M\in\mn,\,g\in  X^*,
$$
respectively. As above, a set $\mathcal{K}$ is compact provided it's restriction to each interval $[0,M]$ is compact in $(L^q(0,M;X),w)$ and  $C([0,M];(X,w))$, respectively.
Furthermore, it can be seen that these topological spaces belong to the class of the so-called quasi-Polish spaces, where the Jakubowski-Skorokhod theorem \cite[Theorem 2]{jakubow} applies. Indeed, in these spaces there exists a countable family of continuous functions that separate points. The proof of tightness of the corresponding laws in the current setting is therefore reduced to exactly the same method as in \cite[Section 4]{BrHo}. Note that the key estimate was \cite[(4.1)]{BrHo}, which is replaced by \eqref{b2xy}. Consequently, the passage to the limit follows the lines of \cite[Section 4]{BrHo}. In addition, Lemma \ref{lem:stac} and Lemma \ref{lem:stac2} show that this limit procedure preserves stationarity  defined in Definition \ref{D2} and Definition \ref{D1}. Hence the limit solution is stationary.
Finally, we obtain \eqref{b2a} by passing to the limit in \eqref{b4a}. The passage to the limit in the stochastic integral can be justified for instance with help of \cite[Lemma 2.1]{debussche1}.
\end{proof}

\begin{Remark}
Note that for $n=1$ the generalized energy inequality \eqref{b2a} corresponds to the usual energy inequality established in \cite{BrFeHo2015A}. The higher order version for $n\in\mn$ is new and employed in order to obtain an analog of Proposition \ref{prop:1608a} suitable for the subsequent limit procedures $\ve\to0$ and $\d\to0$ in Section \ref{L} and Section \ref{P}.
\end{Remark}

\section{Vanishing viscosity limit}
\label{L}

Our goal in this section is to perform the passage to the limit as $\varepsilon\to0$. This represents the most critical and delicate part of our construction. Remark that after completing this limit procedure we have already proved existence of stationary solutions to the stochastic Navier-Stokes system for compressible fluids -- under an additional assumption upon the adiabatic exponent $\gamma$. The last passage to the limit presented in Section \ref{P} is then devoted to weakening this additional assumption.

We point out that the key results needed for the previous limit procedure in Section \ref{d}, namely, Proposition \ref{prop:90} and consequently Proposition \ref{prop:1608a},  depend on $\varepsilon$. Furthermore, it turns out that the global in time energy estimates uniform in $\ve$ and $\d$ are very delicate. On the contrary, in the existence proof in \cite{BrHo}, the basic energy estimate \cite[Proposition 3.1]{BrHo} established on the first approximation level held true uniformly in all the approximation parameters. Consequently, no further manipulations with the energy inequality were needed.
This brought significant technical simplifications in comparison to the present construction of stationary solutions. To be more precise,
 due to the fact that already after the passage to the limit $N\to\infty$, the energy balance is violated and has to be replaced by an inequality. In other words, one cannot justify the application of  It\^o's formula anymore  and it is necessary to establish a more general version of the energy inequality, cf. \eqref{b2a}.

Recall from \cite[Section 5]{BrHo}, that in addition to the usual energy estimate \cite[(5.2)]{BrHo}, a higher integrability of the density \cite[(5.9)]{BrHo} was necessary in order to justify the compactness argument. Nevertheless, as in the deterministic setting it was not possible to obtain strong convergence of the approximate densities directly. Consequently, the identification of the limit proceeded in two steps. First, the passage to the limit in the approximate system was done but the expressions with nonlinear dependence on the density could not be identified. Second, a stochastic analog of the effective viscous flux method originally due to Lions \cite{LI4} allowed to prove strong convergence of the densities and hence to complete the proof.

Let us begin with an estimate for the velocity.

\begin{Proposition}\label{prop:1608}
Let $[\vr_{\varepsilon}, \vu_{\varepsilon}]$ be the stationary solution to \eqref{L1}--\eqref{L2} constructed in Proposition \ref{prop:1608b}. Then  for a.e. $t\in(0,\infty)$
\begin{equation} \label{d1}
\expe{
\| \vu_\varepsilon \|^2_{W^{1,2}(\mathbb T^3; \R^3)} } \leq c(G,M_0),
\end{equation}
\begin{equation} \label{d12a}
\begin{split}
&\expe{\intO{ \left[ \frac{1}{2} \vr_\varepsilon |\vu_\varepsilon|^2 + \frac{a}{\gamma - 1} \vr_\varepsilon^\gamma + \frac{\delta}{\Gamma - 1} \vr_\varepsilon^{\Gamma} \right] }
\left\| \vu_\varepsilon \right\|^2_{W^{1,2} (\mathbb T^3; \R^3)} }\\
&\qquad\leq c(G,M_0)\, \expe{ \intO{ \left[ \frac{1}{2} \vr_\varepsilon |\vu_\varepsilon|^2 + \frac{a}{\gamma - 1} \vr_\varepsilon^\gamma + \frac{\delta}{\Gamma - 1} \vr_\varepsilon^{\Gamma} \right] }  }+c(M_0).
\end{split}
\end{equation}
\end{Proposition}

\begin{proof}
Taking expectation in the energy inequality \eqref{b2a}, we observe that due to stationarity of $[\vr_{\varepsilon}, \vu_{\varepsilon}]$, the first term is constant in time, thus its time derivative vanishes. We recall that $M_\varepsilon\leq c(M_0)$ and using \eqref{m1} we estimate
\begin{align}\label{b4n}
\begin{aligned}
\sum_{k=1}^\infty\intO{ \frac{1}{\vr_\varepsilon} |\varrho_\varepsilon\vc{g}_k(\varrho_\varepsilon,\varrho_\varepsilon\vu_\varepsilon)|^2 }&\leq c(G) \int_{\mt}\varrho_\varepsilon \dx\leq c(G,M_0),
\end{aligned}
\end{align}
and
\begin{equation}\label{b4bc}
\begin{split}
\sum_{k=1}^\infty\left( \intO{ \varrho_\varepsilon\vc{g}_k(\vr_\varepsilon,\vr_\varepsilon\vu_\varepsilon) \cdot \vu_\varepsilon } \right)^2 &\leq \sum_{k=1}^\infty \left\| \sqrt{\varrho_\varepsilon}\vc{g}_k(\vr_\varepsilon,\vr_\varepsilon\vu_\varepsilon) \right\|^2_{L^2(\mathbb T^3;\R^3)} \| \sqrt{\vr_\varepsilon} \vu_\varepsilon \|^2_{L^2(\mathbb T^3; \R^3)} \\
&\leq
c(G,M_0)\intO{ \frac{1}{2} \vr_\varepsilon |\vu_\varepsilon|^2}.
\end{split}
\end{equation}
which leads to
\begin{equation} \label{d51}
\begin{split}
& 2 \ep \,\E\bigg[\bigg(\intO{ \left[ \frac{1}{2} \vr_\varepsilon |\vu_\varepsilon|^2 + \frac{a\gamma }{\gamma - 1} \vr_\varepsilon^\gamma + \frac{\delta \Gamma }{\Gamma - 1} \vr_\varepsilon^{\Gamma} \right] }\bigg)^n\bigg]\\
&\quad + \E\bigg[\left( \intO{ \left[ \frac{1}{2} \vr_{\varepsilon} |\vu_{\varepsilon}|^2 + \frac{a}{\gamma - 1} \vr_{\varepsilon}^\gamma + \frac{\delta}{\Gamma - 1} \vr_{\varepsilon}^{\Gamma} \right] } \right)^{n-1}\intO{ \mathbb{S}(\Grad \vu_\varepsilon) : \Grad \vu_\varepsilon } \bigg]\\
& \leq c(G,M_0)\,\E\bigg[\left( \intO{ \left[ \frac{1}{2} \vr_{\varepsilon} |\vu_{\varepsilon}|^2 + \frac{a}{\gamma - 1} \vr_{\varepsilon}^\gamma + \frac{\delta}{\Gamma - 1} \vr_{\varepsilon}^{\Gamma} \right] } \right)^{n-1}\bigg]\\
&\quad+ \E\bigg[\left( \intO{ \left[ \frac{1}{2} \vr_{\varepsilon} |\vu_{\varepsilon}|^2 + \frac{a}{\gamma - 1} \vr_{\varepsilon}^\gamma + \frac{\delta}{\Gamma - 1} \vr_{\varepsilon}^{\Gamma} \right] } \right)^{n-1}\\
&\qquad\qquad\times H\left( \frac{1}{M_0} \intO{ \vr_\varepsilon } \right) \intO{ \left( \frac{a \gamma}{\gamma - 1} \vr_\varepsilon^{\gamma - 1} + \frac{\delta \Gamma}{\Gamma - 1} \vr_\varepsilon^{\Gamma - 1}
 \right) } \bigg].
\end{split}
\end{equation}
Moreover, it follows as a consequence of Corollary \ref{cor:inv}  that
\[
H \left( \frac{1}{M_0} \intO{ \vr_\varepsilon } \right) = H \left( \frac{M_{\ep}}{M_0} \right) = 2 \ep M_{\ep}.
\]
Hence, setting $n=1$ and applying the Korn--Poincar\'e inequality yields \eqref{d1}, since the second term on the right hand side in \eqref{d51} can be absorbed in the first term on the left hand side. Setting $n=2$ in \eqref{d51} we deduce \eqref{d12a}.
\end{proof}

We point out that  the  corresponding bound for the energy which can be obtained from \eqref{d51}, i.e.
\begin{align*}
\ep  \expe{ \intO{ \left[ \frac{1}{2} \vr_\varepsilon |\vu_{\varepsilon}|^2 + \frac{a\gamma }{\gamma - 1} \vr_{\varepsilon}^\gamma + \frac{\delta \Gamma }{\Gamma - 1} \vr_{\varepsilon}^{\Gamma} \right] } }\leq c(G,M_0)
\end{align*}
still depends on $\varepsilon$ and is therefore not suitable for the passage to the limit $\varepsilon\to 0$. As the next step, we derive an improved estimate for the energy as well as for the pressure.

\begin{Proposition}\label{prop:1708}
Let $[\vr_\varepsilon,\vu_\varepsilon]$ be the stationary solution to \eqref{L1}--\eqref{L2} constructed in Proposition \ref{prop:1608b}. Then the following uniform bound holds true for a.e. $t\in(0,\infty)$
\begin{equation} \label{d131}
\expe{  \intO{ \left[ a \vr_\varepsilon^{\gamma +1} + \delta \vr_\varepsilon^{\Gamma + 1}   +
\frac{1}{3}\vr_\varepsilon^{2} |\vu_\varepsilon|^2 \right]} } \leq c(\delta,G,M_0).
\end{equation}
In addition, if  $s\in \big(1,\frac{\G+1}{\G-1}\wedge\frac{2(\g+1)}{\g+2}\big]$
 then  for a.e.  $T > 0$ and  $\tau>0$
\begin{equation} \label{d151}
\begin{split}
&\expe{ \left( \sup_{t \in [T, T + \tau] } \intO{ \left[ \frac{1}{2} \vr_\varepsilon |\vu_\varepsilon|^2 + \frac{a}{\gamma - 1} \vr_\varepsilon^\gamma + \frac{\delta}{\Gamma - 1} \vr_\varepsilon^{\Gamma} \right] } \right)^s }\\
&\qquad +\expe{\bigg(\int_T^{T+\tau}\|\vu_\ve\|_{W^{1,2}(\mt;\R^3)}^2\dt\bigg)^s}\leq c(\tau,\delta, M_0, G, s),
\end{split}
\end{equation}
where the constant is independent of $T$.
\end{Proposition}

\begin{proof}
Our goal is to use
the quantity
$\Grad\Delta_x^{-1} \left[ \vr_\varepsilon - M_\varepsilon \right]$
as test functions in the momentum equation,
where $\Delta_x$ is the periodic Laplacian. We apply It\^{o}'s formula to the functional $f(\rho,\bfq)=\int_{\mt} \bfq\cdot\Delta_x^{-1}\nabla_x\rho\dx.$
This step can be made rigorous by mollifying the equation, see \cite[Section 5]{BrHo}.
After a rather tedious but straightforward manipulations,
we deduce from (\ref{L1}) and \eqref{L2} that
\begin{equation} \label{d41}
\begin{split}
\int_T^{T+1} &\intO{ \left( a \vr_\varepsilon^{\gamma + 1} + \delta \vr_\varepsilon^{\Gamma + 1} \right) } \ \dt + \int_T^{T+1} \intO{ \frac{1}{3} \vr_\varepsilon^{2} |\vu_\varepsilon|^2 } \ \dt \\
&=
M_\varepsilon \int_T^{T+1}  \intO{ \left( a \vr_\varepsilon^\gamma +
\delta \vr_\varepsilon^\Gamma \right) }   \dt + \frac{1}{3} M_\varepsilon \int_T^{T+1} \intO{ \vr_\varepsilon |\vu_\varepsilon|^2 } \dt\\
 &\qquad+
\int_T^{T+1} \intO{ \left(\frac{4}{3} \mu+\eta \right) \Div \vu_\varepsilon \ \vr_\varepsilon } \dt \\
&\qquad+  2\ep \int_T^{T+1} \intO{ \vre \vue \cdot \Grad \Delta_x^{-1} \left[
\vr_\varepsilon -M_\varepsilon \right] } \dt +\ep \int_T^{T+1} \intO{ \vre^2 \Div \vue } \dt \\
&\qquad - \int_T^{T+1} \intO{ \left( \vr_\varepsilon \vu_\varepsilon \otimes \vu_\varepsilon - \frac{1}{3} \vre |\vu_\varepsilon |^2 \mathbb{I} \right) : \Grad \Delta^{-1}_x
\Grad  \vr_\varepsilon } \dt \\
 &\qquad+
\left[ \intO{ \vr_\varepsilon \vu_\varepsilon \cdot \Grad \Delta^{-1}_x \left[ \vr_\varepsilon - M_\varepsilon \right] } \right]_{t = T}^{t = T+1} \\
&\qquad+ \int_T^{T+1} \intO{ \vr_\varepsilon \vu_\varepsilon \cdot \Grad \Delta_x^{-1} \left[ \Div (\vr_\varepsilon \vu_\varepsilon )  \right] } \dt  \\
&\qquad- \sum_{k=1}^\infty\int_T^{T+1} \intO{ \mathbf{G}_k(\vre,\vre\vue) \cdot \Grad \Delta_x^{-1} \left[ \vr_\varepsilon - M_\varepsilon\right] } \ \D W_k.
\end{split}
\end{equation}
Note that in the above the second term on the left hand side, the second term on the right hand side and the second summand on the fifth line were added artificially and they cancel out. Passing to expectations in (\ref{d41}) and keeping in mind that the processes are stationary we deduce
\begin{equation*}
\begin{split}
&\expe{  \intO{ \left[ a \vr_\varepsilon^{\gamma +1} + \delta \vr_\varepsilon^{\Gamma + 1}   +
\frac{1}{3}\vr_\varepsilon^{2} |\vu_\varepsilon|^2 \right]} }\\
& \leq
c(M_0) \,\expe{  \intO{ \Big( \frac{1}{2} \vr_\varepsilon |\vu_\varepsilon|^2 + \frac{a}{\gamma - 1} \vr_\varepsilon^\gamma + \frac{\delta}{\Gamma - 1} \vr_\varepsilon^\Gamma \Big) } } \\
&\quad - \expe{  \intO{ \Big( \vr_\varepsilon \vu_\varepsilon \otimes \vu_\varepsilon - \frac{1}{3} \vr_\varepsilon |\vu_\varepsilon |^2 \mathbb{I} \Big) : \Grad \Delta^{-1}_x \Grad \vr_\varepsilon  } } \\
&\quad+
\expe{  \intO{ \Big(\frac{4}{3}\mu+\eta \Big) \Div \vu_\varepsilon \ \vr_\varepsilon  }\dt }+ \expe{ \int_T^{T+1} \intO{ \vr_\varepsilon \vu_\varepsilon \cdot \Grad \Delta^{-1}_x  \Div (\vr_\varepsilon \vu_\varepsilon )  }}\\
&\quad+ 2 \ep \,\expe{ \intO{ \vre \vue \cdot \Grad \Delta_x^{-1} \left[
\vr_\varepsilon -M_\varepsilon \right] } } +\ep\,\expe{ \intO{ \vre^2 \Div\vue } }\\
&=:(I)+(II)+(III)+(IV)+(V)+(VI).
\end{split}
\end{equation*}
Now, we estimate each term separately. By Young's inequality we obtain for every $\kappa>0$
\begin{align*}
(I)&\leq\,\kappa\, \expe{ \intO{ \Big( \frac{1}{3} \vr_\varepsilon^2 |\vu_\varepsilon|^2 + a\vr_\varepsilon^{\gamma+1} + \delta\vr_\varepsilon^{\Gamma+1} \Big) } }
+c(\kappa,M_0)\, \expe{  \int_{\mathbb T^3}\Big( |\vu_\varepsilon|^2 + 1 \Big) \dx }\\
&\leq\,\kappa\, \expe{  \intO{ \Big( \frac{1}{3} \vr_\varepsilon^2 |\vu_\varepsilon|^2 + a\vr_\varepsilon^{\gamma+1} + \delta\vr_\varepsilon^{\Gamma+1} \Big) } }+c(\kappa,G,M_0),
\end{align*}
using the uniform bound (\ref{d1}).
In order to control the remaining integrals on the right hand side, we first use H\" older's inequality to obtain
\begin{equation*}
\begin{split}
(II)
&\leq \,c\, \expe{\| \sqrt{\vr_\varepsilon} \vu_\varepsilon \|_{L^2(\mathbb T^3; \R^3)} \| \vu_\varepsilon \|_{L^6(\mathbb T^3; \R^3)} \left\| \sqrt{\vr_\ve} \,\Grad \Delta_x^{-1} \Grad \vr_\varepsilon \right\|_{L^3(\mathbb T^3; \R^{3\times 3})}} \\
&\leq c
\left( \expe{\| \sqrt{\vr_\varepsilon} \vu_\varepsilon \|^2_{L^2(\mathbb T^3; \R^3)} \| \vu_\varepsilon \|^2_{W^{1,2}(\mathbb T^3; \R^3)}} +  \expe{\left\| \sqrt{\vr_\varepsilon}\, \Grad \Delta^{-1}_x \Grad \vr_\varepsilon \right\|^2_{L^{3 }(\mathbb T^3; \R^{3\times 3})}}  \right).
\end{split}
\end{equation*}
Furthermore, since $\Gamma\geq9/2$, we have
\begin{align*}
&\expe{\left\| {\vr_\varepsilon}^{1/2} \Grad \Delta^{-1}_x \Grad\vr_\varepsilon \right\|^2_{L^3(\mathbb T^3; \R^{3 \times 3})}}
\leq \expe{\left\| \vr_\varepsilon \right\|_{L^{\frac{9}{2}}(\mathbb T^3)} \left\| \Grad \Delta^{-1}_x \Grad \vr_\varepsilon \right\|^2_{L^\frac{9}{2}(\mathbb T^3; \R^{3 \times 3})}}\\
&\leq\,c\,\expe{\left\| \vr_\varepsilon \right\|^3_{L^{\frac{9}{2}}(\mathbb T^3)}}\dt\leq\,c\,\mathbb E \bigg[\left\| \vr_\varepsilon \right\|^3_{L^{\Gamma}(\mathbb T^3)}\bigg]\leq\,\kappa\delta\expe{\left\| \vr_\varepsilon \right\|^\Gamma_{L^{\Gamma}(\mathbb T^3)}}+c(\kappa,\delta).
\end{align*}
Note that we also used the continuity of $\Grad \Delta^{-1}_x \Grad$
and Young's inequality.
Similarly, we can estimate
\begin{align*}
(IV) &\leq\expe{
\|\vu_\varepsilon \|_{L^6(\mathbb T^3; \R^3)} \| \vr_\varepsilon \|_{L^{3}(\mathbb T^3)} \left\| \Grad \Delta^{-1}_x \Div \left( \vr_\varepsilon\vu_\varepsilon  \right)
\right\|_{L^2(\mathbb T^3; \R^{3 \times 3})}} \\
&\leq\,c\,\expe{
\|\vu_\varepsilon \|_{L^6(\mathbb T^3; \R^3)} \| \vr_\varepsilon \|_{L^{3}(\mathbb T^3)} \left\|\vr_\varepsilon\vu_\varepsilon
\right\|_{L^2(\mathbb T^3; \R^{3})}}\\
&\leq\,c\,\expe{
\|\vu_\varepsilon \|_{W^{1,2}(\mathbb T^3; \R^3)}^2 \| \vr_\varepsilon \|_{L^{3}(\mathbb T^3; \R^3)}^2 }\\
&\leq\,c\,\expe{
\|\vu_\varepsilon \|_{W^{1,2}(\mathbb T^3; \R^3)}^2 \| \vr_\varepsilon \|_{L^{\Gamma}(\mathbb T^3; \R^3)}^\Gamma}+c(G,M_0)
\end{align*}
using (\ref{d1}). We also have that
\begin{align*}
(III)&\leq\,\kappa\delta\,\expe{\|\varrho_\varepsilon\|_{L^2(\mt)}^2}+c(\kappa,\delta)\expe{\|\nabla\bfu_\varepsilon\|^2_{L^2(\mt;\R^3)}}\\
&\leq\,\kappa\delta\,\expe{\|\varrho_\varepsilon\|_{L^\G(\mt)}^\Gamma}+c(\kappa,\delta,G,M_0)
\end{align*}
as well as
\begin{align*}
(VI)&\leq\,\kappa\delta\,\E\|\varrho_\varepsilon\|_{L^\G(\mt)}^\Gamma+c(\kappa,\delta,G,M_0).
\end{align*}
Finally, continuity of $\nabla_x\Delta_x^{-1}$ and (\ref{d1}) imply
\begin{align*}
(V)&\leq\,\kappa\delta\,\bigg(\E\|\varrho_\varepsilon\|_{L^4(\mt)}^4+\E\|\nabla_x\Delta_x^{-1}[\varrho_\varepsilon-M_\varepsilon]\|_{L^4(\mt)}^4\dx\bigg)+c(\kappa,\delta)\|\bfu_\varepsilon\|_{L^2(\mt;\R^3)}^2\\
&\leq\,c\,\kappa\delta\,\E\|\varrho_\varepsilon\|_{L^\G(\mt)}^\Gamma+c(\kappa,\delta,G,M_0)
\end{align*}
Summing up the inequalities above, choosing $\kappa$ small enough and using stationarity, we obtain
\begin{equation*}
\begin{split}
\mathbb E\bigg[& \intO{ a \vr_\varepsilon^{\gamma + 1}  + \delta \vr_\varepsilon^{\Gamma +1}} + \intO{ \frac{1}{3} \vr_\varepsilon^{2} |\vu_\varepsilon|^2 } \bigg] \\
&\leq \expe{ \intO{ \left[ \frac{1}{2} \vr_\varepsilon |\vu_\varepsilon|^2 + \frac{a}{\gamma - 1} \vr_\varepsilon^\gamma + \frac{\delta}{\Gamma - 1} \vr_\varepsilon^\Gamma  \right] } \,\| \vu_\varepsilon \|_{W^{1,2}(\mt;\R^3)}^2  } + c(\delta,G,M_0).
\end{split}
\end{equation*}
Thus, due to \eqref{d12a} and  Young's inequality, we may conclude that the stationary solution $[\vr_\ve,\vu_\ve]$  admits the uniform bound \eqref{d131} as well as
\begin{equation} \label{d141}
\expe{ \left(1 + \intO{ \left[ \frac{1}{2} \vr_\varepsilon |\vu_\varepsilon|^2 + \frac{a}{\gamma - 1} \vr_\varepsilon^\gamma + \frac{\delta}{\Gamma - 1} \vr_\varepsilon^\Gamma \right] } \right) \| \vu_\varepsilon \|^2_{W^{1,2} (\mathbb T^3; \R^3)} } \leq c(\delta,G,M_0).
\end{equation}

Finally, let us show \eqref{d151}.
To this end, we may go back to the energy inequality (\ref{b2a}) for $n=1$, obtaining
\[
\begin{split}
&\expe{ \left( \sup_{t \in [T, T + \tau] } \intO{ \left[ \frac{1}{2} \vr_\varepsilon |\vu_\varepsilon|^2 + \frac{a}{\gamma - 1} \vr_\varepsilon^\gamma + \frac{\delta}{\Gamma - 1} \vr_\varepsilon^{\Gamma} \right] } \right)^s }
+\expe{\bigg(\int_T^{T+\tau}\|\vu_\ve\|_{W^{1,2}(\mt;\R^3)}^2\dt\bigg)^s}\\
 &\qquad\leq c(s)\,\expe{ \sup_{t \in [T, T + \tau] } \left|\sum_{k=1}^\infty\int_{T}^{t} \intO{ \vr_\ve\,\mathbf{g}_k(\vr_\ve,\vr_\ve\vu_\ve) \cdot \vu_\varepsilon } \,\D W_k \right|^s }\\
 &\qquad\quad +c(s)\,\expe{\sup_{t \in [T, T + \tau] } \bigg|\sum_{k=1}^\infty\int_T^{t}\int_{\mt}\frac{1}{\varrho_\varepsilon}|{\varrho_\varepsilon}\mathbf g_k({\varrho_\varepsilon},{\varrho_\varepsilon}\vu_\varepsilon)|^2\,\dif x\,\dif r\bigg|^s}\\
 &\qquad\quad+ c(s)\,\expe{\sup_{t \in [T, T + \tau] } \bigg|\int_T^{t} \int_{\mt}\vr_\varepsilon^{\gamma-1}+\vr_\varepsilon^{\Gamma-1}\,\dif x\,\dif r\bigg|^s}.
\end{split}
\]
The first term on the right hand side is estimated using the Burkholder-Davis-Gundy inequality and  (\ref{b4bc}); the second term using \eqref{b4n}. We deduce that
\[
\begin{split}
&\expe{ \left( \sup_{t \in [T, T + \tau] } \intO{ \left[ \frac{1}{2} \vr_\varepsilon |\vu_\varepsilon|^2 + \frac{a}{\gamma - 1} \vr_\varepsilon^\gamma + \frac{\delta}{\Gamma - 1} \vr_\varepsilon^{\Gamma} \right] } \right)^s } \\
& \quad\leq c(s,\tau,M_0,G)\left(1+\E\bigg[\bigg(\int_T^{T+\tau}\int_{\mt}\vr_\varepsilon|\vu_\varepsilon|^2\,\dif x\,\dif r\bigg)^\frac{s}{2}\bigg] + \expe{ \bigg|\int_T^{T+\tau} \int_{\mt}\vr_\varepsilon^{\gamma-1}+\vr_\varepsilon^{\Gamma-1}\,\dif x\,\dif r\bigg|^s}\right).
\end{split}
\]
Now, by H\" older's inequality, stationarity, \eqref{d141} and \eqref{d131}, for $s\in (1,2)$,
\begin{equation}\label{e4}
\begin{split}
&\E\left( \int_T^{T+\tau}\intO{ \vr_\varepsilon |\vu_\varepsilon |^2 } \,\dif r\right)^\frac{s}{2}\\
&\qquad\leq c(\tau,s)\,\E\bigg(\int_T^{T+\tau}\left( \| \sqrt{\vr_\varepsilon} \vu_\varepsilon \|_{L^2(\mathbb T^3; \R^3)} \| \vu_\varepsilon \|_{L^6(\mathbb T^3; \R^3)} \right)^s \| \sqrt{\vr_\ve} \|_{L^{3}(\mathbb T^3)}^s\dt\bigg)^\frac{1}{2}\\
&\qquad\leq c(\tau,s)\, \E\bigg(\int_T^{T+\tau} \| \sqrt{\vr_\varepsilon} \vu_\varepsilon \|_{L^2(\mathbb T^3; \R^3)}^2 \| \vu_\varepsilon \|_{L^6(\mathbb T^3; \R^3)}^2+\|\sqrt{\vr_\ve}\|_{L^3(\mt)}^\frac{2s}{2-s}\dt\bigg)^\frac{1}{2}\\
&\qquad\leq c(\tau,s) \left(\expe{ \| \sqrt{\vr_\varepsilon} \vu_\varepsilon \|^2_{L^2(\mathbb T^3; \R^3)} \| \vu_\varepsilon \|_{W^{1,2}(\mathbb T^3; \R^3)}^2 }+ \E\|\vr_\varepsilon\|_{L^\g(\mt)}^{\frac{s}{2 -{s}}}\right)\\
&\qquad\leq c(\tau,s,\d,G,M_0)\left(1+ \E\left( \intO{ \vr_\varepsilon^{\gamma+1} } \right)^{\frac{s}{(\gamma+1)( 2 -{s})}} \right)\\
&\qquad\leq c(\tau,s,\d,G,M_0)\left(1+ \E\intO{ \vr_\varepsilon^{\gamma+1} }  \right)\leq c(\tau,s,\d,G,M_0)
\end{split}
\end{equation}
provided $s\leq \tfrac{2(\gamma+1)}{\gamma+2}$. Similarly,
\[
\bigg(\int_{\mt}\vr_\varepsilon^{\gamma-1}+\vr_\varepsilon^{\Gamma-1}\,\dif x\bigg)^s\leq 1+\int_{\mt}\vr_\varepsilon^{\gamma+1}+\vr_\varepsilon^{\Gamma+1}\,\dif x,
\]
provided $s\leq\tfrac{\Gamma+1}{\Gamma-1}$. Consequently, \eqref{d151} follows due to \eqref{d131}.
\end{proof}

With Proposition \ref{prop:1608} and Proposition \ref{prop:1708} at hand, we are able to follow the compactness argument of  \cite[Section 5.1]{BrHo}.
To be more precise, as $\varepsilon\rightarrow0$ we aim at constructing stationary solutions to the following system.
\begin{itemize}
\item {\bf Equation of continuity.}
\begin{equation} \label{L1'}
\begin{split}
\int_{0}^{\infty} \intO{ \left[ \vr \partial_t \varphi + \vr \vu \cdot \Grad \varphi \right] } \dt  =0
\end{split}
\end{equation}
for any $\varphi \in \DC((0, \infty) \times \mt)$ $\prst$-a.s.
\item {\bf Regularized momentum equation.}
\begin{equation} \label{L2'}
\begin{split}
\int_0^\infty \partial_t \psi\intO{  \vr \vu \cdot \varphi} \dt  &+ \int_0^\infty \psi\intO{  \vr \vu \otimes \vu : \Grad \varphi } \dt
+ \int_0^\infty \psi\intO{ (a  \vr^\gamma+\delta\varrho^\Gamma)  \Div \varphi } \dt  \\
&-   \int_0^\infty \psi\intO{  \mathbb{S}(\Grad \vu) : \Grad \varphi} \dt
=-\sum_{k=1}^\infty \int_0^\infty \psi \intO{ \mathbf{G}_k(\vr,\vr\vu) \cdot \varphi } \ \D W_k
\end{split}
\end{equation}
for any $\psi \in \DC((0, \infty))$, $\varphi \in C^\infty(\mt; \R^3)$ $\prst$-a.s.

\end{itemize}

Note that unlike the energy estimate in \cite{BrHo}, the bound \eqref{d151} only gives limited moment estimates, i.e. $s$ cannot be arbitrarily large. Nevertheless, \eqref{d151} is sufficient to perform the passage to the limit. We also point out that the assumption \eqref{m1} on the noise coefficients is actually stronger than the one in \cite{BrHo}, and consequently the convergence of the stochastic integral is more straightforward.

We deduce the following.

\begin{Proposition}\label{prop:1608b'}
Let $\delta>0$ be given. Then there exists a stationary solution $[\varrho_\delta,\bfu_\delta]$ to \eqref{L1'}--\eqref{L2'}. Moreover, we have the estimates
\begin{equation} \label{d1'}
\expe{
\| \vu_\delta(t) \|^2_{W^{1,2}(\mathbb T^3; \R^3)} } \leq c(G,M_0),
\end{equation}
and
\begin{equation} \label{d51'}
\begin{split}
&\expe{ \left( \intO{ \left[ \frac{1}{2} \vr_{\delta} |\vu_{\delta}|^2 + \frac{a}{\gamma - 1} \vr_{\delta}^\gamma + \frac{\delta}{\Gamma - 1} \vr_{\delta}^{\Gamma} \right] } \right)
 \intO{ \| \vu_{\delta}\|^2_{W^{1,2}(\mt;\R^3)} } }\\
&\leq c(G,M_0)
\expe{
\intO{ \left[ \frac{1}{2} \vr_{\delta} |\vu_{\delta}|^2 + \frac{a}{\gamma - 1} \vr_{\delta}^\gamma + \frac{\delta}{\Gamma - 1} \vr_{\delta}^{\Gamma} \right] } } + c(M_0),
\end{split}
\end{equation}
for a.e. $t\in(0,\infty)$.
In addition, the equation of continuity \eqref{L1'} holds true in the renormalized sense and for all $\psi\in C^\infty_c ((0,\infty))$, $\psi\geq0$, the following  energy inequality holds true
\begin{align} \label{b2ac}
\begin{aligned}
&-\int_0^\infty\partial_t\psi \left( \intO{ \left[ \frac{1}{2} \vr_\d |\vu_\d|^2 + \frac{a}{\gamma - 1} \vr_\d^\gamma + \frac{\delta}{\Gamma - 1} \vr_\d^{\Gamma} \right] } \right)\dt\\
&\quad+  \int_0^\infty\psi
 \intO{ \mathbb{S}(\Grad \vu_\d) : \Grad \vu_\d } \dt \\
& \leq \sum_{k=1}^\infty \int_0^\infty\psi \intO{ \varrho_\d\,\bfg_k(\varrho_\d,\vr_\d\vu_\d) \cdot \vu_\d } \, \D W_k \\
&\quad+ \frac{1}{2} \int_0^\infty\psi
\sum_{k=1}^\infty\intO{ \frac{1}{\vr_\d} |\varrho_\d\,\vc{g}_k(\varrho_\d,\vr_\d\vu_\d)|^2 } \dt.
\end{aligned}
\end{align}
\end{Proposition}

\begin{proof}
First, we proceed as in \cite[Section 5.1]{BrHo} and establish the necessary tightness of the joint law of $[\vr_\ve,\vr_\ve\vue,\vue,W]$. The only difference is that the corresponding path spaces are replaced by their local-in-time analogs as discussed in the proof of Proposition \ref{prop:1608b}. Consequently, the Jakubowski-Skorokhod theorem applies and we obtain a new family of martingale solutions, still denoted by
$[\vr_\ve,\vr_\ve\vue,\vue,W]$, obeying the same laws and converging in probability with respect to a new basis, still denoted by
$\bas$.
In addition, the limit satisfies
\begin{equation} \label{L3}
\int_{0}^{\infty} \intO{ \left[ \vr \partial_t \varphi + \vr \vu \cdot \Grad \varphi \right] } \dt = 0, \qquad \intO{ \vr } = M_0,
\end{equation}
for any $\varphi \in \DC((0, \infty) \times \mt)$ $\prst$-a.s.,
\begin{equation} \label{L4}
\begin{split}
\int_0^\infty \partial_t \psi \intO{ \vr \vu \cdot \varphi} \dt  &+ \int_0^\infty\psi \intO{  \vr \vu \otimes \vu : \Grad \varphi } \dt
+ \int_0^\infty \psi\intO{ \big(a  \Ov{\vr^\gamma} +\d \Ov{\vr^\G}\big) \Div \varphi } \dt  \\
&-   \int_0^\infty \psi\intO{  \mathbb{S}(\Grad \vu) : \Grad \varphi} \dt  =- \int_0^\infty \psi  \D M_\varphi
\end{split}
\end{equation}
for any $\psi \in \DC((0, \infty))$, $\varphi \in C^\infty(\mt; \R^3)$ $\prst$-a.s. Here $M_\varphi$ is a square integrable martingale and the bars denote the corresponding weak limits with respect to $t,x$. For details, we refer to \cite[Proposition 5.6]{BrHo}.
In addition, $\varrho$ satisfies the renormalized equation of continuity. That is
\begin{align}\label{eq:rendelta}
\partial_t b({\varrho})&+\Div\big(b({\varrho}){\bfu}\big)
+\big(b'({\varrho}){\varrho}-b({\varrho})\big)\Div{\bfu}
=0
\end{align}
in the sense of distribution on $(0,\infty)\times\mt$ for every $b\in C^1([0,\infty))$ with $b'(z)=0$ for $z\geq M_b$ for some constant $M_b>0$. However, as discussed in \cite[Remark 1.1]{feireisl1}, the assumption on $b'$ can be weakened to
$$|b'(z)z|\leq c(z^\theta+z^\frac{\g}{2})\quad\text{for all } z>0 \text{ and some }\theta\in(0,\tfrac{\g}{2}).$$
This in  particular includes the function $b(z)=z\log z$ employed below.

In order to complete the proof, it is enough to show strong convergence of the densities as in \cite[Section 5.2]{BrHo}. More specifically, we prove that 
\begin{equation} \label{L5}
\limsup_{\ep \to 0} \expe{ \| \vre - \vr \|^{\G + 1}_{L^{\G + 1}(\mathbb T^3)} }\leq
\limsup_{\ep \to 0} \expe{ \intO{ \left( \vre^{\G + 1} - \Ov{\vr^\G} \vr \right) } } \leq 0 \ \mbox{for any} \ t > 0.
\end{equation}
Note that the first inequality follows from the algebraic inequality which holds true
\[
(A - B)^{\Gamma + 1} = (A - B)^\Gamma (A-B) \leq\, (A^\Gamma - B^\Gamma)(A - B) \ \mbox{whenever}\ A,B \geq 0.
\]
In order to see {the rightmost inequality} in (\ref{L5}) we use the method of Lions \cite{LI4} based on regularity of the effective viscous flux. More specifically, mimicking
the technique from the proof of Proposition \ref{prop:1708}, we derive from \eqref{L1}--\eqref{L2} the identity
\begin{equation}\label{L6a}
\begin{split}
\int_T^{T+1} &\intO{ \Big(a \vre^{\gamma + 1}+\delta \vre^{\Gamma + 1}\Big) } \ \dt  =M_\varepsilon
 \int_T^{T+1}  \intO{ \left( a \vre^\gamma+\delta \vre^{\Gamma } \right) }    \dt \\
&+ \int_T^{T+1} \intO{\Big( \vre \vue \cdot \Grad \Delta^{-1}  \Div (\vre \vue ) - \vre \vue \otimes \vue : \Grad \Delta^{-1}_x
\Grad \vre  \Big)} \dt \\
&+
\int_T^{T+1} \intO{ \left(\frac{4}{3} \mu+\eta \right) \Div \vue \ \vre } \dt \\
&+ 2\ep \int_T^{T+1} \intO{ \vre \vue \cdot \Grad \Delta^{-1} \left[
\vre - M_\varepsilon \right] } \dt + \ep \int_T^{T+1} \intO{ \vre^2 \Div \vue } \dt \\
&+
\left[ \intO{ \vr \vu \cdot \Grad \Delta^{-1}_x \left[ \vre - M_\varepsilon \right] } \right]_{t = T}^{t = T+1} \\
&- \sum_{k=1}^\infty\int_T^{T+1} \intO{ \mathbf{G}_k(\vre,\vre\vue) \cdot \Grad \Delta^{-1} \left[ \vre -  M_\varepsilon  \right] } \ \D W_k.
\end{split}
\end{equation}
In addition, since $\vre$ satisfies the equation of continuity in the strong sense, the application of the commutator lemma in the spirit of \cite{DL} yields
\begin{align*}
\dif (\vre\log\vre)=-\Div\big(\vre\log\vre\,{\bfu}_\varepsilon\big)-\vre\,\Div\vue+\varepsilon\Delta( \vre\log\vre)-\varepsilon \frac{|\nabla\vre|^2}{\vre}.
\end{align*}
Inserting this into \eqref{L6a} implies
\begin{equation} \label{L6}
\begin{split}
\int_T^{T+1} &\intO{ \Big(a \vre^{\gamma + 1}+\delta \vre^{\Gamma + 1}\Big) } \ \dt  =
 M_\varepsilon\int_T^{T+1}\intO{  \Big(a \vre^\gamma+\delta \vre^{\Gamma}\Big)  }  \dt \\
&+ \int_T^{T+1} \intO{\Big( \vre \vue \cdot \Grad \Delta^{-1}  \Div (\vre \vue ) - \vre \vue \otimes \vue : \Grad\Delta^{-1}_x
\Grad \vre  \Big)} \dt \\
&
- \left( \frac{4}{3} \mu+\eta \right) \left[  \intO{ \vre \log (\vre) } \right]_{t = T}^{t = T+1} - \ep  \left( \frac{4}{3} \mu+\eta \right)\int_T^{T+1} \intO{ \frac{|\Grad \vre|^2}{\vre} } \dt \\
&+  2\ep \int_T^{T+1} \intO{ \vre \vue \cdot \Grad \Delta^{-1} \left[
\vre - M_\varepsilon \right] } \dt + \ep \int_T^{T+1} \intO{ \vre^2 \Div \vue } \dt\\
&+
\left[ \intO{ \vre \vue \cdot \Grad \Delta^{-1}_x \left[ \vre - M_\varepsilon \right] } \right]_{t = T}^{t = T+1} \\
&- \sum_{k=1}^\infty\int_T^{T+1} \intO{ \vc{G}_k(\vre,\vre\vue) \cdot \Grad \Delta^{-1} \left[ \vre - M_\varepsilon\right] } \ \D W_k.
\end{split}
\end{equation}
Similarly, as the limit density $\vr$ also satisfies the renormalized equation of continuity \eqref{eq:rendelta}, we deduce choosing $b(z)=z\log z$ that
\begin{align*}
\dif (\vr\log\vr)=-\Div\big(\vr\log\vr\,{\bfu}\big)-\vr\,\Div\vu
\end{align*}
holds true in the sense of distributions.
Therefore, we obtain from the limit equations (\ref{L3}), (\ref{L4}) that\begin{equation*}
\begin{split}
\int_T^{T+1}& \intO{ \left(a \Ov{\vr^{\gamma}} +\delta\Ov{\vr^\Gamma} \right)\vr  } \ \dt  =M_0
 \int_T^{T+1}  \intO{ \left( a \Ov{\vr^\gamma}+\delta\Ov{\vr^\Gamma} \right) } \dt \\
&+ \int_T^{T+1} \intO{ \Big(\vr \vu \cdot \Grad \Delta^{-1}  \Div (\vr \vu ) - \vr \vu \otimes \vu : \Grad\Delta^{-1}_x
\Grad\vr \Big) } \dt \\
&
- \left( \frac{4}{3} \mu+\eta \right) \left[  \intO{ \vr \log \vr} \right]_{t = T}^{t = T+1} \\
&+
\left[ \intO{ \vr \vu \cdot \Grad \Delta^{-1}_x \left[ \vr - M_0 \right] } \right]_{t = T}^{t = T+1} - \int_T^{T+1} \D M_\Phi,
\end{split}
\end{equation*}
with
\[
\Phi = \Grad \Delta^{-1}_x \left[ \vr -  M_0\right].
\]

Thus passing to expectations and using the fact that the processes are stationary, we get
\begin{align}
\nonumber
\E\bigg[ &\int_T^{T+1}\intO{\Big(a \vre^{\gamma + 1}+\delta\vre^\Gamma\Big)\vre  } \dt\bigg]   \leq
M_{\ep} \,\expe{ \int_T^{T+1}\intO{  \Big(a \vre^\gamma+\delta\vre^\Gamma\Big)  } \dt} \\
\label{2702}&+ \expe{ \int_T^{T+1}\intO{\Big( \vre \vue \cdot \Grad \Delta^{-1}  \Div (\vre \vue ) - \vre \vue \otimes \vue : \Grad\Delta^{-1}_x
\Grad \vre \Big) } \dt } \\
&+ 2 \ep \,\expe{  \intO{ \vre \vue \cdot \Grad \Delta^{-1} \left[
\vre - M_\varepsilon \right] } } + \ep \,\expe{ \intO{ \vre^2 \Div \vue } }.\nonumber
\end{align}
Note that the inequality is due to the fact that we are not able to pass to the limit in the fourth term on the right hand side of \eqref{L6} and we can only use its negativity.
Similarly we obtain
\begin{align}\nonumber
&\expe{ \int_T^{T+1}\intO{ \Big(a \Ov{\vr^{\gamma}}+\delta\Ov{\vr^\Gamma}\Big) \vr  }\dt }  =
M_0\, \expe{ \int_T^{T+1}\intO{ \left( a \Ov{\vr^\gamma}+\delta\Ov{\vr^\Gamma} \right) } \dt } \\
&+ \expe{ \int_T^{T+1} \intO{ \Big(\vr \vu \cdot \Grad \Delta^{-1}  \Div (\vr \vu ) - \vr \vu \otimes \vu : \Grad\Delta^{-1}_x
\Grad \vr \Big) }\dt }.\label{2702b}
\end{align}
Note that the $\varepsilon$-terms in \eqref{2702} vanish due to Proposition \ref{prop:1708} and we have $M_\varepsilon\rightarrow M_0$ as $\varepsilon\to0$.
Consequently, the desired conclusion (\ref{L5}) follows as soon as we observe that
\begin{equation} \label{L8}
\begin{split}
&\lim_{\ep \to 0} \expe{ \int_T^{T+1} \intO{ \Big(\vre \vue \cdot \Grad \Delta^{-1}  \Div (\vre \vue ) - \vre \vue \otimes \vue : \Grad \Delta^{-1}_x
\Grad  \vre \Big) } \dt }\\
&= \expe{ \int_T^{T+1} \intO{ \Big(\vr \vu \cdot \Grad \Delta^{-1}  \Div (\vr \vu ) - \vr \vu \otimes \vu : \Grad \Delta^{-1}_x
\Grad  \vr \Big) } \dt }.
\end{split}
\end{equation}
In fact, \eqref{L8} in combination with \eqref{2702} and \eqref{2702b} implies
\begin{align*}
\limsup_{\ve\to0}\,\expe{ \int_T^{T+1}\intO{\Big(a \vre^{\gamma }+\delta\vre^\Gamma\Big)\vre  } \dt}\leq \expe{\int_T^{T+1} \intO{ \Big(a \Ov{\vr^{\gamma}}+\delta\Ov{\vr^\Gamma}\Big) \vr  } \dt}
\end{align*}
which shows strong convergence of $\vre$ by monotonicity arguments.
Relation \eqref{L8} can be established by compensated compactness arguments (applied $\p$-a.s.) if we show that the expressions under expectations are $\prst$-equi-integrable.
Considering the two summands separately and using continuity of $\Grad \Delta^{-1}_x \Grad$, we have
\[
\begin{split}
\left| \intO{ \vre \vue \cdot \Grad \Delta^{-1}  \Div (\vre \vue ) } \right| &\leq \,
c \| \sqrt{\vre} \vue \|_{L^2(\mathbb T^3)} \| \sqrt{\vre} \|_{L^{2 \Gamma}(\mathbb T^3)}  \| \Grad \Delta^{-1}_x \Grad\vre\vue \|_{L^{\frac{2\Gamma}{\Gamma-1}}(\mathbb T^3; \R^3)}\\
&\leq \,
c \| \sqrt{\vre} \vue \|_{L^2(\mathbb T^3)} \| \sqrt{\vre} \|_{L^{2 \Gamma}(\mathbb T^3)}  \| \vre\vue \|_{L^{\frac{2\Gamma}{\Gamma-1}}(\mathbb T^3; \R^3)}\\
&\leq
c \| \sqrt{\vre} \vue \|_{L^2(\mathbb T^3)}  \| \sqrt{\vre} \|_{L^{2 \Gamma}(\mathbb T^3)}\| \vue \|_{L^{6}(\mathbb T^3; \R^3)} \| \vre \|_{L^\Gamma (\mathbb T^3) }\\
&\leq
c \| \sqrt{\vre} \vue \|_{L^2(\mathbb T^3)} \| \sqrt{\vre} \|_{L^{2 \Gamma}(\mathbb T^3)} \| \vue \|_{W^{1,2}(\mathbb T^3; \R^3)} \| \vre \|_{L^\Gamma (\mathbb T^3) },
\end{split}
\]
as $\Gamma \geq \frac{9}{2}$. Similarly, we have
\[
\begin{split}
\left| \intO{ \vre \vue \otimes \vue : \Grad \Delta^{-1}_x \Grad \vre } \right| &\leq
c \| \sqrt{\vre} \vue \|_{L^2(\mathbb T^3)} \| \vu \|_{W^{1,2}(\mathbb T^3; \R^3)} \| \sqrt{\vre} \|_{L^{2 \Gamma}(\mathbb T^3)} \| \vre \|_{L^\Gamma (\mathbb T^3) }.
\end{split}
\]
Here, in accordance with (\ref{d141}),
\[
\expe{ \| \sqrt{\vre} \vue \|^2_{L^2(\mathbb T^3)} \| \vu \|^2_{W^{1,2}(\mathbb T^3; \R^3)} } \leq c(\delta,G,M_0),
\]
while, by virtue of (\ref{d131}),
\[
\| \sqrt{\vre} \|_{L^{2 \Gamma}(\mathbb T^3)} \| \vre \|_{L^\Gamma (\mathbb T^3) } = \| \vre \|_{L^\Gamma (\mathbb T^3)}^{\frac{3}{2}} \in L^q(\Omega),\qquad
q = \frac{2 \Gamma }{3} > 2.
\]

We have shown (\ref{L5}); whence strong convergence of $\vre$. Consequently, as in \cite[Section 5.2]{BrHo}, we may identify the nonlinear terms in \eqref{L4} and hence $[\vr,\vu]$ is a weak martingale solution to \eqref{L1'}--\eqref{L2'}. Stationarity then follows by Lemma \ref{lem:stac} and Lemma \ref{lem:stac2}. The estimate \eqref{d1'} and \eqref{d51'}, respectively, is obtained by weak lower semicontinuity from  \eqref{d1} and \eqref{d12a}, respectively, since the constants were uniform in $\ve$. The same arguments give the energy inequality \eqref{b2ac}. Note that the passage to the limit in the stochastic integral can be justified for instance with help of \cite[Lemma 2.1]{debussche1}.
\end{proof}

\begin{Remark}\label{rem:1808}
It is important to note that there is an essential difference between the strong convergence of the density in the existence theory, see \cite[Section 5.2]{BrHo}, and the above proof. More specifically, the existence theory requires compactness of the initial data which is not available in the present setting. Instead the fact that the solution is stationary must be used.
\end{Remark}

\section{Vanishing artificial pressure limit}
\label{P}

As the final step of the proof of our main result, Theorem \ref{Tm1}, it remains to perform the last limit procedure, that is, $\d\to0$. Recall that according to  Proposition \ref{prop:1608b'}, the stationary solutions constructed in the previous section already satisfy the uniform bounds \eqref{d1'} and \eqref{d51'}. Nevertheless, the pressure estimate as well as the estimate for the energy and velocity from Proposition \ref{prop:1708} all blow up as $\d$ vanishes. Therefore, in order to apply the compactness argument from \cite[Section 6]{BrHo} it is necessary to improve these estimates. The rest of the construction then proceeds exactly as in \cite[Section 6.1--6.3]{BrHo}.

\begin{Proposition}\label{prop:1808}
Let $[\vr_\d,\vu_\d]$ be the stationary solution to \eqref{L1'}--\eqref{L2'} constructed in Proposition~\ref{prop:1608b'}. Then the following uniform bound holds true for some $\a>0$ and  a.e. $t\in(0,\infty)$
\begin{equation} \label{d13}
\expe{  \intO{ \bigg[ a \vr_{\delta}^{\gamma + \alpha} + \delta \vr_{\delta}^{\Gamma + \alpha}   +
\vr_{\delta}^{1 + \alpha} |\vu_{\delta}|^2 \bigg]} } \leq c(G,M_0),
\end{equation}
In addition, for some   $s>1$
 and for a.e.  $T > 0$ and  $\tau>0$
\begin{equation} \label{d151'}
\begin{split}
&\expe{ \left( \sup_{t \in [T, T + \tau] } \intO{ \left[ \frac{1}{2} \vr_\d |\vu_\d|^2 + \frac{a}{\gamma - 1} \vr_\d^\gamma + \frac{\delta}{\Gamma - 1} \vr_\d^{\Gamma} \right] } \right)^s }\\
&\qquad +\expe{\bigg(\int_T^{T+\tau}\|\vu_\d\|_{W^{1,2}(\mt;\R^3)}^2\dt\bigg)^s}\leq c(\tau, M_0, G, s),
\end{split}
\end{equation}
where the constant is independent of $T$.

\end{Proposition}

\begin{proof}
As far as the pressure estimates are concerned we use the test function
$$\Grad\Delta_x^{-1} \bigg[\varrho^\alpha-\int_{\mathbb T^3}\varrho^\alpha\dx\bigg],\quad \alpha>0.$$
We obtain
after a rather tedious but straightforward manipulation the following analogue of \eqref{d41}
\begin{equation} \label{d2}
\begin{split}
&\int_T^{T+1} \intO{ \left( a \vr_\delta^{\gamma + \alpha} + \delta \vr_\delta^{\Gamma + \alpha} \right) } \ \dt + \int_T^{T+1} \intO{ \frac{1}{3} \vr_\delta^{1 + \alpha} |\vu|^2 } \ \dt \\
&=
 \int_T^{T+1} \left( \intO{ \left( a \vr_\delta^\gamma +
\delta \vr_\delta^\Gamma \right) } \intO{ \vr_\delta^\alpha } \right) \dt + \frac{1}{3}  \int_T^{T+1}
\left( \intO{ \vr_\delta |\vu_\delta|^2 } \dx \intO{ \vr_\delta^\alpha } \right) \dt\\
&\qquad+
\int_T^{T+1} \intO{ \left(\frac{4}{3} \mu+\eta \right) \Div \vu_\delta \ \vr_\delta^\alpha } \dt \\
&\qquad - \int_T^{T+1} \intO{ \left( \vr_\delta \vu_\delta \otimes \vu_\delta - \frac{1}{3} \vr_\delta |\vu_\delta |^2 \mathbb{I} \right) : \nabla_x \Delta^{-1}_x
\nabla_x \left[ \vr_\delta^\alpha \right] } \dt \\
&\qquad+
\left[ \intO{ \vr_\delta \vu_\delta \cdot \Grad \Delta^{-1}_x \left[ \vr_\delta^\alpha - \intO{\vr_\delta^\alpha } \right] } \right]_{t = T}^{t = T+1} \\
&\qquad- \int_T^{T+1} \intO{ \vr_\delta \vu_\delta \cdot \Grad \Delta^{-1} [ \D (\vr_\delta^\alpha) ] } \\
&\qquad- \sum_{k=1}^\infty\int_T^{T+1} \intO{ \vc{G}_k(\vr_\delta,\vr_\delta \vu_\delta) \cdot \Grad \Delta^{-1} \left[ \vr_\delta^\alpha -  \intO{\vr_\delta^\alpha } \right] } \ \D W_k.
\end{split}
\end{equation}
Next, we evoke the renormalized equation of continuity \eqref{eq:rendelta}
\begin{equation*}
\D \vr_\delta^{\alpha} + \Div (\vr_\delta^\alpha \vu_\delta ) \dt + (\alpha - 1) \vr_\delta^\alpha \Div \vu_\delta  \dt = 0
\end{equation*}
deducing from (\ref{d2})
\begin{equation} \label{d4}
\begin{split}
&\int_T^{T+1} \intO{ \left( a \vr_\delta^{\gamma + \alpha} + \delta \vr_\delta^{\Gamma + \alpha} \right) } \ \dt + \int_T^{T+1} \intO{ \frac{1}{3} \vr_\delta^{1 + \alpha} |\vu_\delta|^2 } \ \dt \\
&=
 \int_T^{T+1} \left( \intO{ \left( a \vr_\delta^\gamma +
\delta \vr^\Gamma \right) } \intO{ \vr_\delta^\alpha } \right) \dt + \frac{1}{3} \int_T^{T+1}
\left( \intO{ \vr_\delta |\vu_\delta|^2 } \intO{ \vr_\delta^\alpha } \right) \dt\\
&\qquad+
\int_T^{T+1} \intO{ \left(\frac{4}{3} \mu+\eta \right) \Div \vu_\delta \ \vr_\delta^\alpha } \dt \\
&\qquad - \int_T^{T+1} \intO{ \left( \vr_\delta \vu_\delta \otimes \vu_\delta - \frac{1}{3} \vr |\vu_\delta |^2 \mathbb{I} \right) : \Grad \Delta^{-1}_x
\Grad \left[ \vr_\delta^\alpha \right] } \dt \\
&\qquad+
\left[ \intO{ \vr_\delta \vu_\delta \cdot \Grad \Delta^{-1}_x \left[ \vr_\delta^\alpha -  \intO{\vr_\delta^\alpha } \right] } \right]_{t = T}^{t = T+1} \\
&\qquad+ \int_T^{T+1} \intO{ \vr_\delta \vu_\delta \cdot \Grad \Delta_x^{-1} \left[ \Div (\vr_\delta^\alpha \vu_\delta ) + (\alpha - 1) \vr_\delta^\alpha \Div \vu_\delta  \right] } \dt  \\
&\qquad- \sum_{k=1}^\infty\int_T^{T+1} \intO{ \vc{G}_k(\vr_\d,\vr_\d\vu_\d) \cdot \Grad \Delta^{-1}_x \left[ \vr_\delta^\alpha -  \intO{\vr_\delta^\alpha } \right] } \ \D W_k.
\end{split}
\end{equation}

Before proceeding, we make the assumption that $0 < \alpha < 1/3$, which implies in particular
\begin{equation*}
\left| \intO{ \vr_\delta^\alpha } \right| \leq c(M_0), \qquad \left\| \Grad\Delta_x^{-1} \left[ \vr_\delta^\alpha -  \intO{\vr_\delta^\alpha } \right] \right\|_{L^\infty(\mathbb T^3;\R^3)} \leq c(M_0)
\end{equation*}
using H\"older's inequality, Sobolev's embedding and continuity of $\Grad\Delta_x^{-1}\Grad$.
Passing to expectations in (\ref{d4}) and keeping in mind that the processes are stationary we deduce
\begin{equation} \label{Dd5}
\begin{split}
&\expe{  \intO{ \bigg[ a \vr_{\delta}^{\gamma + \alpha} + \delta \vr_{\delta}^{\Gamma + \alpha}   +
\vr_{\delta}^{1 + \alpha} |\vu_{\delta}|^2 \bigg]} }  \\
& \leq
c(M_0) \left( \expe{ \intO{ \left[ \frac{1}{2} \vr_\delta |\vu_\delta|^2 + \frac{a}{\gamma - 1} \vr^\gamma + \frac{\delta}{\Gamma - 1} \vr_\delta^\Gamma \right] } } + 1 \right) \\
 &\qquad+
\expe{ \intO{ \left( \frac{4}{3}\mu+\eta \right) \Div \vu_\delta \ \vr_\delta^\alpha  } } \\
& \qquad+ \expe{ \intO{ \left( \vr_\delta \vu_\delta \otimes \vu_\delta - \frac{1}{3} \vr_\delta |\vu_\delta |^2 \mathbb{I} \right) : \Grad \Delta^{-1}_x \Grad \left[ \vr_\delta^\alpha \right] } } \\
&\qquad + \expe{ \intO{ \vr_\delta \vu_\delta \cdot \Grad \Delta^{-1}_x [ \Div (\vr_\delta^\alpha \vu_\delta ) + (\alpha - 1) \vr_\delta^\alpha \Div \vu_\delta ] }}.
\end{split}
\end{equation}
Moreover, using the uniform bound (\ref{d1'}) we may further reduce (\ref{Dd5}) to
\begin{equation*}
\begin{split}
&\expe{  \intO{ \bigg[ a \vr_{\delta}^{\gamma + \alpha} + \delta \vr_{\delta}^{\Gamma + \alpha}   +
\vr_{\delta}^{1 + \alpha} |\vu_{\delta}|^2 \bigg]} } \\
& \leq \expe{ \intO{ \left( \vr_\delta \vu_\delta \otimes \vu_\delta - \frac{1}{3} \vr_\delta |\vu_\delta |^2 \mathbb{I} \right) : \Grad \Delta^{-1}_x \Grad \left[ \vr_\delta^\alpha \right] } } \\
&\qquad+ \expe{ \intO{ \vr_\delta \vu_\delta \cdot \Grad \Delta^{-1}_x [ \Div (\vr_\delta^\alpha \vu_\delta ) + (\alpha - 1) \vr_\delta^\alpha \Div \vu_\delta ] }} + c(G,M_0).
\end{split}
\end{equation*}
Note that we applied Young's inequality to the first and second term on the right-hand side of \eqref{Dd5} and in order to absorb the arising term eventually.
To control the remaining integrals on the right hand side, we first use H\" older's inequality to obtain
\begin{equation*}
\begin{split}
&\left| \expe{ \intO{ \left( \vr_\delta \vu_\delta \otimes \vu_\delta - \frac{1}{3} \vr_\delta |\vu_\delta |^2 \mathbb{I} \right) : \Grad \Delta^{-1}_x \Grad \left[ \vr_\delta^\alpha \right] } } \right|\\
&\leq c\,\expe{ \| \sqrt{\vr_\delta} \vu_\delta \|_{L^2(\mathbb T^3; \R^3)} \| \vu_\delta \|_{L^6(\mathbb T^3; \R^3)} \left\| \sqrt{\vr_\d}\, \Grad \Delta^{-1} \Grad \left[ \vr_\delta^\alpha \right] \right\|_{L^3(\mathbb T^3; \R^{3\times 3})} }\\
&\leq c
\left( \expe{\| \sqrt{\vr} \vu_\delta \|^2_{L^2(\mathbb T^3; \R^3)} \| \vu_\delta \|^2_{W^{1,2}(\mathbb T^3; \R^3)}} +\expe{  \left\| \sqrt{\vr_\d}\, \Grad \Delta^{-1} \Grad \left[ \vr_\delta^\alpha \right] \right\|^2_{L^{3 }(\mathbb T^3; \R^{3\times 3})} } \right).
\end{split}
\end{equation*}
Furthermore, we have
\[
\left\| \sqrt{\vr_\d}\, \Grad \Delta^{-1}_x \Grad \left[ \vr_\delta^\alpha \right] \right\|^2_{L^3(\mt; \R^{3 \times 3})}
\leq \left\| \sqrt{\vr_\delta} \right\|^2_{L^{2 \gamma}(\mathbb T^3)} \left\| \Grad \Delta^{-1}_x \Grad \left[ \vr_\delta^\alpha \right] \right\|^2_{L^q(\mathbb T^3; \R^{3 \times 3})}
\]
\[
\frac{1}{2 \gamma} + \frac{1}{q} = \frac{1}{3}, \ \gamma > \frac{3}{2}.
\]
Now, we choose $\alpha > 0$ so small that $\alpha q \leq 1$ to conclude that
\begin{equation*}
\left\| \sqrt{\vr_\d}\, \Div \Grad \Delta^{-1}_x \Grad \left[ \vr_\delta^\alpha \right] \right\|^2_{L^3(\mathbb T^3; \R^{3 \times 3})}
\leq c(M_0) \|\vr_\d\|_{L^\g(\mt)}	.
\end{equation*}

Similarly, we can handle
\begin{equation*}
\begin{split}
&\left|  \intO{ \vr_\delta \vu_\delta \cdot \Grad \Delta^{-1}_x \Div [ \vr_\delta^\alpha \vu_\delta ] } \right|\\
& \leq
\| \sqrt{\vr_\delta} \vu_\delta \|_{L^2(\mathbb T^3; \R^3)} \| \sqrt{\vr_\delta} \|_{L^{2 \gamma}(\mathbb T^3)} \left\| \Grad \Delta^{-1}_x \Grad \left[ \vr_\delta^\alpha \vu_\delta  \right]
\right\|_{L^q(\mathbb T^3; \R^{3 \times 3})}
\end{split}
\end{equation*}
where
\[
\frac{1}{2} + \frac{1}{2 \gamma} + \frac{1}{q} = 1, \ \mbox{in particular}\ q < 6 \ \mbox{if}\ \gamma > \frac{3}{2},
\]
and where
\begin{equation*}
\left\| \Grad \Delta^{-1}_x \Grad \left[ \vr_{\delta}^\alpha \vu  \right]
\right\|_{L^q(\mathbb T^3; \R^{3\times 3})} \leq \| \vr_{\delta}^\alpha \vu_\delta \|_{L^q(\mathbb T^3; \R^3)} \leq \| \vu_\delta \|_{L^6(\mathbb T^3; \R^3)} \| \vr_{\delta}^\alpha \|_{L^s(\mathbb T^3)},
\ \frac{1}{6} + \frac{1}{s} = \frac{1}{q}.
\end{equation*}
Taking $\alpha s \leq 1$ we get, similarly to the above,
\begin{equation*}
\begin{split}
&\left|  \intO{ \vr_{\delta} \vu_\delta \cdot \Grad \Delta^{-1}_x \Div [ \vr_{\delta}^\alpha \vu_\delta  ] } \right|\\
& \leq c (M_0) \left( \| \sqrt{\vr_{\delta}} \vu_\delta \|^2_{L^2(\mathbb T^3; \R^3)}
\| \vu_\delta \|^2_{W^{1,2}(\mathbb T^3; \R^3)} + \| \vr_{\delta} \|_{L^\gamma(\mathbb T^3; \R^3)} \right).
\end{split}
\end{equation*}
Finally,
\[
\begin{split}
&\left| \intO{\vr_{\delta} \vu_\delta \cdot \Grad \Delta^{-1}_x [ \vr_{\delta}^\alpha \Div \vu_\delta ]  }    \right|\\
&\leq \| \sqrt{\vr_{\delta}} \|_{L^{2 \gamma}(\mathbb T^3; \R^3)} \| \sqrt{\vr_{\delta}} \vu_\delta \|_{L^2(\mathbb T^3; \R^3)}
\left\| \Grad \Delta^{-1}_x [\vr_{\delta}^\alpha \Div \vu_\delta ] \right\|_{L^q(\mathbb T^3; \R^3)} \\
& \leq
\frac{1}{2} \left( \| \vr_{\delta} \|_{L^\gamma(\mathbb T^3)} + \| \sqrt{\vr_{\delta} \vu} \|^2_{L^2(\mathbb T^3;\R^3)} \left\| \Grad \Delta^{-1}_x [\vr_{\delta}^\alpha \Div \vu_\delta ] \right\|_{L^q(\mathbb T^3; \R^3)}^2\right),
\end{split}
\]
where
\[
\frac{1}{2\gamma} + \frac{1}{2} + \frac{1}{q} = 1, \ \ q < 6 \ \mbox{if}\ \gamma > \frac{3}{2}.
\]
As the $\Grad \Delta_x^{-1}$-operator gains one derivative, we get, by means of the standard Sobolev embedding,
\[
\left\| \Grad \Delta_x^{-1}[\vr_{\delta}^\alpha \Div \vu_\delta ] \right\|_{L^q(\mathbb T^3; \R^3)} \leq \| \vr_{\delta}^\alpha \Div \vu_\delta \|_{L^r(\mathbb T^3)}, \ r < 2.
\]
Thus, similarly to the previous steps, we may conclude that
\begin{equation*}
\left| \intO{\vr_{\delta} \vu_{\delta} \cdot \Grad [ \vr_{\delta}^\alpha \Div \vu_{\delta} ]  }    \right|
\leq c (M_0) \left( \| \sqrt{\vr_{\delta}} \vu_{\delta} \|^2_{L^2(\mathbb T^3; \R^3)}
\| \vu_{\delta} \|^2_{W^{1,2}_0(\mathbb T^3; \R^3)} + \| \vr_{\delta} \|_{L^\gamma(\mathbb T^3; \R^3)} \right).
\end{equation*}

Summing up the above estimates we obtain 
\begin{equation*}
\begin{split}
&\expe{ \intO{ a \vr_{\delta}^{\gamma + \alpha}  + \delta \vr_{\delta}^{\Gamma + \alpha}+\frac{1}{3} \vr_{\delta}^{1 + \alpha} |\vu_{\delta}|^2 } } \\
&\leq \expe{ \intO{ \left[ \frac{1}{2} \vr_{\delta} |\vu_{\delta}|^2 + \frac{a}{\gamma - 1} \vr_{\delta}^\gamma + \frac{\delta}{\Gamma - 1} \vr_{\delta}^\Gamma  \right] } \intO{ \|  \vu_{\delta} \|_{W^{1,2}(\mt;\R^3)}^2 } } + c(G,M_0),
\end{split}
\end{equation*}
where we absorbed the term $\| \vr_{\delta} \|_{L^\gamma(\mathbb T^3; \R^3)}$ in the left-hand side.
We close the estimates by evoking (\ref{d51'}).
Thus we may conclude that any global in time stationary solutions admit the uniform bound \eqref{d13} as well as
\begin{equation} \label{d14}
\expe{ \left(1 + \intO{ \left[ \frac{1}{2} \vr_{\delta} |\vu_{\delta}|^2 + \frac{a}{\gamma - 1} \vr_{\delta}^\gamma + \frac{\delta}{\Gamma - 1} \vr_{\delta}^\Gamma \right] } \right) \| \vu_{\delta} \|^2_{W^{1,2} (\mathbb T^3; \R^3)} } \leq c(G,M_0).
\end{equation}

Finally, we claim that
\[
\expe{ \left( \intO{ \left[ \frac{1}{2} \vr_{\delta} |\vu_{\delta}|^2 + \frac{a}{\gamma - 1} \vr_{\delta}^\gamma + \frac{\delta}{\Gamma - 1} \vr_{\delta}^{\Gamma} \right] } \right)^s } \leq
c
\]
for a certain $s = s(\alpha) > 1$. Indeed the $\vr_{\delta}$-dependent terms can be estimated directly by (\ref{d13})
while, by H\" older inequality and Sobolev embedding,
\[
\begin{split}
\left( \intO{ \vr_{\delta} |\vu_{\delta} |^2 } \right)^s &\leq \left( \| \sqrt{\vr_{\delta}} \vu_{\delta} \|_{L^2(\mathbb T^3; \R^3)} \| \vu_{\delta} \|_{L^6(\mathbb T^3; \R^3)} \right)^s \| \sqrt{\vr_{\delta}} \|_{L^{3}(\mathbb T^3)}^s\\
&\leq c \left[ \| \sqrt{\vr_{\delta}} \vu_{\delta} \|^2_{L^2(\mathbb T^3; \R^3)} \| \vu_{\delta} \|_{W^{1,2}(\mathbb T^3; \R^3)}^2 + \left( \intO{ \vr_{\delta}^\gamma } \right)^{\frac{1}{\gamma( 2 -s)}} \right].
\end{split}
\]
Note that we also took into account $\gamma>\frac{3}{2}$. This can be estimated by \eqref{d14} provided $s<2-\frac{1}{\gamma}$. The term with $\vr_{\delta}^\gamma$ (and $\delta\vr_{\delta}^\Gamma$) is estimated by Jensen's inequality.
Now we go back to the energy inequality (\ref{b2ac}). Due to \eqref{b4n} we obtain after taking the power $s$ and the supremum in time and expectation
\[
\begin{split}
&\expe{ \left( \sup_{t \in [T, T + \tau] } \intO{ \left[ \frac{1}{2} \vr_{\delta} |\vu_{\delta}|^2 + \frac{a}{\gamma - 1} \vr_{\delta}^\gamma + \frac{\delta}{\Gamma - 1} \vr_{\delta}^{\Gamma} \right] } \right)^s }+\expe{\bigg(\int_T^{T+\tau}\|\vu_\d\|_{W^{1,2}(\mt;\R^3)}^2\dt\bigg)^s} \\ &\leq c( G,M_0)  + \expe{ \sup_{t \in [T, T + \tau] } \left|\sum_{k=1}^\infty\int_{T}^{t}  \intO{ \vc{G}_k(\vr_\d,\vr_\d\vu_\d) \cdot \vu_{\delta} } \, \D W_k \right|^s } .
\end{split}
\]
{Note that all terms are well-defined by \eqref{d14}.}
Here, the second term on the right hand side is controlled by (\ref{b4bc}) and the Burkholder-Davis-Gundy inequality similarly to \eqref{e4} as follows
\begin{align*}
&\expe{ \sup_{t \in [T, T + \tau] } \left|\sum_{k=1}^\infty\int_{T}^{t}  \intO{ \vc{G}_k(\vr_\d,\vr_\d\vu_\d) \cdot \vu_{\delta} } \, \D W_k \right|^s }\\
&\leq c(s,M_0)\,\expe{\bigg(\int_T^{T+\tau} \bigg(\intO{\frac{1}{2}\vr_\d|\vu_\d|^2}\bigg)^\frac{s}{2}\dt},
\end{align*}
which can be again estimated by \eqref{d13}.
We therefore conclude that \eqref{d151'} holds true.
for a.e. $T > 0$, where the constant depends on $\tau$ but it is independent of $T$.
\end{proof}

Finally, we have all in hand in order to complete the proof of Theorem \ref{Tm1}.

\begin{proof}[Proof of Theorem \ref{Tm1}]
We follow the lines of \cite[Section 6]{BrHo}. In view of Proposition \ref{prop:1808}, we are able to apply the Jakubowski-Skorokhod representation theorem and obtain convergence of $[\vr_\d,\vu_\d]$ (in fact, we obtain a new family of martingale solutions defined on a new probability space but keep the original notation for simplicity) to a stationary weak martingale solution of
\begin{equation*}
\int_{0}^{\infty} \intO{ \left[ \vr \partial_t \varphi + \vr \vu \cdot \Grad \varphi \right] } \dt = 0, \qquad \intO{ \vr } = M_0,
\end{equation*}
for any $\varphi \in \DC((0, \infty) \times \mt)$ $\prst$-a.s.,
\begin{equation*}
\begin{split}
\int_0^\infty \partial_t \psi \intO{ \vr \vu \cdot \varphi} \dt  &+ \int_0^\infty\psi \intO{  \vr \vu \otimes \vu : \Grad \varphi } \dt
+ \int_0^\infty \psi\intO{ a  \Ov{\vr^\gamma}  \Div \varphi } \dt  \\
&-   \int_0^\infty \psi\intO{  \mathbb{S}(\Grad \vu) : \Grad \varphi} \dt  =- \int_0^\infty \psi  \D M_\varphi
\end{split}
\end{equation*}
for any $\psi \in \DC((0, \infty))$, $\varphi \in C^\infty(\mt; \R^3)$ $\prst$-a.s. Here $M_\varphi$ is a square integrable martingale and the bars denote the corresponding weak limits.
In addition, $\varrho$ satisfies the renormalized equation of continuity.

In order to identify the nonlinear density dependent terms, we keep Remark \ref{rem:1808} in mind and apply the effective viscous flux method  in the same way as in \cite[Section 6.1--6.3]{BrHo}, which then completes the proof.
Note that similarly to Section \ref{L}, even the limited moment estimates from Proposition \ref{prop:1808} are sufficient for the passage to the limit.
\end{proof}

\appendix

\section{Auxiliary results}

In this final section, we collect several auxiliary results concerning the two notions of stationarity introduced in Definition \ref{D2} and Definition \ref{D1}. 
First of all, we observe that  it is actually enough to consider Definition \ref{D1} for $q=1$.

\begin{Lemma}\label{lem:s}
Let $k\in\mn_0$, $p,q\in[1,\infty)$. If $\bfU$ is stationary  on $L^1_{\mathrm{loc}}([0,\infty);W^{k,p}(\mt))$ in the sense of Definition \ref{D1} and $\bfU\in L^q_{\mathrm{loc}}([0,\infty);W^{k,p}(\mt))$ $\p$-a.s. then $\bfU$ is stationary on $L^q_{\mathrm{loc}}([0,\infty);W^{k,p}(\mt))$.
\end{Lemma}

\begin{proof}
According to the assumption, for all $f\in C_b(L^1_{\text{loc}}([0,\infty);W^{k,p}(\mt)))$, it holds
$$
\E[f(\bfU)]=\E[f(\mathcal{S}_\tau\bfU)].
$$
If $f\in C_b(L^q_{\text{loc}}([0,\infty);W^{k,p}(\mt)))$ then for all $R\in\mn$
$$
\bfU\mapsto f(\bfU\,\mathbf{1}_{|\bfU|\leq R})\in C_b(L^1_{\text{loc}}([0,\infty);W^{k,p}(\mt)))
$$
hence
$$
\E[f(\bfU\,\mathbf{1}_{|\bfU|\leq R})]=\E[f((\mathcal{S}_\tau\bfU)\mathbf{1}_{|\mathcal{S}_\tau\bfU|\leq R})].
$$
Finally, since $\bfU\in L^q_{\mathrm{loc}}([0,\infty);W^{k,p}(\mt))$ $\p$-a.s., we obtain that
$$
\bfU\,\mathbf{1}_{|\bfU|\leq R}\to\bfU\quad\text{in}\quad L^q_{\mathrm{loc}}([0,\infty);W^{k,p}(\mt))\quad\p\text{-a.s.}
$$
and we conclude by the dominated convergence.
\end{proof}

Next, we show that for the case of stochastic processes with  continuous trajectories, the two definitions are equivalent.

\begin{Lemma}\label{l:equivD12}
Let $k\in\mn_0$, $p\in[1,\infty)$. An $W^{k,p}(\mt)$-valued measurable stochastic process $\bfU$ with $\p$-a.s. continuous trajectories is stationary on $W^{k,p}(\mt)$ in the sense of Definition \ref{D2} if and only if it is stationary on $L^1_{\mathrm{loc}}([0,\infty);W^{k,p}(\mt))$ in the sense of Definition \ref{D1}.
\end{Lemma}

\begin{proof}
Let us first show that Definition \ref{D1} implies Definition \ref{D2}. Let $\tau\geq0$ and $t_1,\dots,t_n\in [0,\infty)$. Let $(\psi_m)$ be a {smooth and compactly supported} approximation to the identity on $\R$ and define
$$\Psi_m(\bfU)= \left(\int_0^\infty \bfU(s)\psi_m(t_1-s)\dif s,\dots, \int_0^\infty \bfU(s)\psi_m(t_n-s)\dif s\right).$$
If $\varphi\in C_b([W^{k,p}(\mt)]^n)$ then $\varphi\circ\Psi_m\in C_b(L^1_{\text{loc}}([0,\infty);W^{k,p}(\mt)))$ and therefore
$$
\E[\varphi\circ\Psi_m(\mathcal{S}_\tau\bfU)]=\E[\varphi\circ\Psi_m(\bfU)].
$$
Sending $m\to\infty$ we obtain due to the continuity of $\bfU$ and the dominated convergence theorem that
$$
\E[\varphi(\bfU(t_1+\tau),\dots,\bfU(t_n+\tau))]=\E[\varphi(\bfU(t_1),\dots,\bfU(t_n))]
$$
and the claim follows.

To show the converse implication, let us fix $\tau\geq0$ and an equidistant partition $0=t_1<\cdots<t_n<\cdots<\infty$ with mesh size $\Delta t=\frac{\tau}{m}$ for some $m\in\mn$. Observe that there is an one-to-one correspondence between sequences $\hat\bfU_m=(\bfU(t_1),\bfU(t_2),\dots)\in \ell^1_{\rm loc}(W^{k,p}(\mt))$  and piecewise constant functions in $L^1_{\rm loc}([0,\infty);W^{k,p}(\mt))$
given by $\tilde\bfU_m(t)=\bfU(t_i)$ if $t\in[t_i,t_{i+1})$. Moreover, it is an isometry in the following sense
$$
\sum_{i=1}^N\|\hat\bfU_m(t_i)\|_{W^{k,p}(\mt)}=\int_0^{N\Delta t}\|\tilde \bfU_m(t)\|_{W^{k,p}(\mt)}\,\dif t.
$$
Thus, if $\Phi$ denotes this isometry and
$\varphi\in C_b (L^1_{\rm loc}([0,\infty);W^{k,p}(\mt)))$, then $\varphi\circ\Phi\in C_b(\ell^1_{\rm loc}(W^{k,p}(\mt)))$. Consequently,
$$
\E[\varphi(\tilde\bfU_m)]=\E[\varphi(\mathcal{S}_\tau\tilde\bfU_m)]
$$
follows from Definition \ref{D2}.
Due to the continuity of $\bfU$ we may send $m\to\infty$ which completes the proof.
\end{proof}

The following result proves that weak continuity together with a uniform bound is enough for the equivalence of Definition \ref{D2} and Definition \ref{D1} to hold true.

\begin{Corollary}\label{l:equivD123}
The statement of Lemma \ref{l:equivD12} remains valid if the trajectories of $\bfU$ are $\p$-a.s. weakly continuous and for all $T>0$
\begin{equation}\label{eq:gh}
\sup_{t\in[0,T]}\|\bfU\|_{W^{k,p}(\mt)}<\infty\quad\p\text{-a.s.}
\end{equation}
\end{Corollary}

\begin{proof}
Let $(\varphi_\varepsilon)$ be an approximation to the identity on $\mt$. Since $\bfU$ has weakly continuous trajectories in $W^{k,p}(\mt)$ and satisfies \eqref{eq:gh}, the process $\bfU^\varepsilon:=\bfU*\varphi_\varepsilon$ has strongly continuous trajectories in $W^{k,p}(\mt)$. Hence the equivalence of the two notions of stationarity from Lemma \ref{l:equivD12} holds.

Now, let $\bfU$ be stationary on $L^1_{\text{loc}}([0,\infty);W^{k,p}(\mt))$ in the sense of Definition \ref{D1}. That is, for every $f\in C_b(L^1_{\text{loc}}([0,\infty);W^{k,p}(\mt)))$ we have
$$\E[f(\mathcal{S}_\tau\bfU)]=\E[f(\bfU)].$$
Since $\bfU\mapsto f(\bfU*\varphi_\varepsilon)$ also belongs to $C_b(L^1_{\text{loc}}([0,\infty);W^{k,p}(\mt)))$ we deduce that
$$\E[f(\bfU^\varepsilon)]=\E[f([\mathcal{S}_\tau\bfU]*\varphi_\varepsilon)]=\E[f(\mathcal{S}_\tau \bfU^\varepsilon)].$$
So, $\bfU^\varepsilon$ is stationary in the sense of Definition \ref{D1} and due to Lemma \ref{l:equivD12}, $\bfU^\varepsilon$ is stationary in the sense of Definition \ref{D2}. In addition, $\bfU^\varepsilon(t)\to\bfU(t)$ strongly in $W^{k,p}(\mt)$ for every $t\in[0,\infty)$. Therefore, if $g\in C_b([W^{k,p}(\mt)]^n)$, we may use dominated convergence in order to pass to the limit in  expressions of the form
$$\E[g(\bfU^\varepsilon(t_1),\dots, \bfU^\varepsilon(t_n))]=\E[g(\bfU^\varepsilon(t_1+\tau),\dots, \bfU^\varepsilon(t_n+\tau))].$$
Stationarity of $\bfU$ in the sense of Definition \ref{D2} follows.

To show the converse implication, assume that $\bfU$ is stationary in the sense of Definition \ref{D2}. By the same argument as above, it follows that $\bfU^\varepsilon$ is stationary in the sense of Definition \ref{D2} hence stationary in the sense of Definition \ref{D1}. In other words, for every $f\in C_b(L^1_{\text{loc}}([0,\infty);W^{k,p}(\mt)))$,
$$\E[f(\bfU^\varepsilon)]=\E[f(\mathcal{S}_\tau\bfU^\varepsilon)].$$
According to \eqref{eq:gh} we obtain that $\bfU^\varepsilon\to \bfU$ in $L^1_{\text{loc}}([0,\infty);W^{k,p}(\mt))$ and the dominated convergence theorem yields the claim.
\end{proof}

As the next step, we show that both notions of stationarity introduced in Definition \ref{D2} and Definition~\ref{D1} are stable under weak convergence.

\begin{Lemma}\label{lem:stac}
Let $k\in\mn_0, p,q\in[1,\infty)$ and let $(\bfU_m)$ be a sequence of random variables taking values in $L^q_{\rm loc}([0,\infty);W^{k,p}(\mt)))$. If, for all $m\in\mn$, $\bfU_m$ is stationary on $L^q_{\rm loc}([0,\infty);W^{k,p}(\mt))$ in the sense of Definition \ref{D1} and
\begin{align*}
\bfU_m\rightharpoonup \bfU\quad\text{in}\quad L^q_{\mathrm{loc}}([0,\infty);W^{k,p}(\mt))\quad\mathbb P\text{-a.s.,}
\end{align*}
then $\bfU$ is stationary on $L^q_{\rm loc}([0,\infty);W^{k,p}(\mt))$.
\end{Lemma}

\begin{proof}
Stationarity of $\bfU_m$ implies that for every $f\in C_b(L^q_{\rm loc}([0,\infty);W^{k,p}(\mt)))$ and every $\tau\geq0$
\begin{equation}\label{eq:str}
\E[ f(\mathcal{S}_\tau\bfU_m)]=\E[  f(\bfU_m)].
\end{equation}
Moreover, it follows from the above weak convergence and the weak continuity of
$$\mathcal{S}_\tau:L^q_{\rm loc}([0,\infty);W^{k,p}(\mt)))\to L^q_{\rm loc}([0,\infty);W^{k,p}(\mt)))$$
that for every $g\in C_b((L^q_{\rm loc}([0,\infty);W^{k,p}(\mt)),{w}))$
it holds
$$g(\mathcal{S}_\tau\bfU_m)\to g(\mathcal{S}_\tau\bfU),\qquad g(\bfU_m)\to g(\bfU).$$
In particular, since  every weakly continuous function is strongly continuous hence \eqref{eq:str} holds with $f$ replaced by $g$, we deduce by the dominated convergence theorem that
$$\E[ g(\mathcal{S}_\tau\bfU)]=\E[  g(\bfU)].$$
Now, it remains to verify the corresponding expression for a general strongly continuous function $f\in C_b(L^q_{\rm loc}([0,\infty);W^{k,p}(\mt)))$. To this end, let $(\varphi_\varepsilon)$ be a smooth approximation to the identity on $\R\times\mt$. Since convolution with $\varphi_\varepsilon$ is a compact operator on $L^q_{\rm loc}([0,\infty);W^{k,p}(\mt))$, we obtain that 
$$\bfU\mapsto f(\bfU*\varphi_\varepsilon)=:f(\bfU^\varepsilon)\in C_b((L^q_{\rm loc}([0,\infty);W^{k,p}(\mt)),{w}))$$
and consequently
$$\E[  f(\bfU^\varepsilon)]=\E[ f([\mathcal{S}_\tau\bfU]*\varphi_\varepsilon)]=\E[f(\mathcal{S}_\tau\bfU^\varepsilon)],$$
hence $\bfU^\varepsilon$ is stationary.
Since
$$\bfU^\varepsilon\to\bfU\quad\text{in}\quad L^q_{\mathrm{loc}}([0,\infty);W^{k,p}(\mt))\quad\mathbb P\text{-a.s.,} $$
we may pass to the limit $\varepsilon\to0$ and conclude using the dominated convergence theorem.
\end{proof}

\begin{Lemma}\label{lem:stac2}
Let $k\in\mn_0$, $p\in[1,\infty)$ and let $(\bfU_m)$ be a sequence of $W^{k,p}(\mt)$-valued  stochastic processes which are stationary on $W^{k,p}(\mt)$ in the sense of Definition \ref{D2}. If for all $T>0$
\begin{equation}\label{eq:gh1}
\sup_{m\in\mn}\expe{\sup_{t\in[0,T]}\|\bfU_m\|_{W^{k,p}(\mt)}}<\infty
\end{equation}
and
\begin{align*}
\bfU_m\rightarrow \bfU\quad\text{in}\quad C_{\mathrm{loc}}([0,\infty);(W^{k,p}(\mt),w))\quad\mathbb P\text{-a.s.,}
\end{align*}
then $\bfU$ is stationary on $W^{k,p}(\mt)$.
\end{Lemma}

\begin{proof}
The claim is a consequence of Corollary \ref{l:equivD123} and Lemma \ref{lem:stac}. Indeed, as a consequence of \eqref{eq:gh1} we deduce that
$$\expe{\sup_{t\in[0,T]}\|\bfU_m\|_{W^{k,p}(\mt)}}<\infty$$
thus $\bfU_m$ satisfies the assumptions of Corollary \ref{l:equivD123} and the same is true for $\bfU$ due to lower semicontinuity of the corresponding norm. Accordingly, $\bfU_m$ satisfy the assumptions of Lemma \ref{lem:stac} which implies that $\bfU$ is stationary in the sense of Definition \ref{D1}. Corollary \ref{l:equivD123} then yields the claim.
\end{proof}

Let us conclude with a simple observation that stationarity is preserved under composition with measurable functions.

\begin{Corollary}\label{cor:00}
Let $k\in\mn_0$, $p\in[1,\infty)$. Let the stochastic process $\bfU$ be stationary on $W^{k,p}(\mt)$ in the sense of Definition \ref{D2}. Then for every measurable function $F:W^{k,p}(\mt)\to\R$, the stochastic process $F(\bfU)$ is stationary on $\R$.
\end{Corollary}

\begin{proof}
The proof follows immediately from the corresponding equality of joint laws of $(\bfU(t_1),\dots, \bfU(t_n))$ and $(\bfU(t_1+\tau),\dots, \bfU(t_n+\tau))$.
\end{proof}

\begin{Corollary}\label{cor:001}
Let $k\in\mn_0$, $p,q\in[1,\infty)$. Let $\bfU$ be stationary on $L^q_{\mathrm{loc}}([0,\infty);W^{k,p}(\mt))$ in the sense of Definition \ref{D1}. Then for every measurable function $F:W^{k,p}(\mt)\to\R$  and a.e. $s,t\in [0,\infty)$, the laws of
$\bfU(s)$ and $\bfU(t)$ on $W^{k,p}(\mt)$ coincide.
\end{Corollary}

\begin{proof}
Since the mapping $\bfU\mapsto \bfU(t) \mapsto F(\bfU(t))$ is measurable on $L^q_{\text{loc}}([0,\infty);W^{k,p}(\mt))$ for a.e. $t\in[0,\infty)$. For the same reasons, the mapping $\mathcal{S}_{s-t}:\bfU\mapsto \bfU(s) \mapsto F(\bfU(s))$ is measurable on $L^q_{\text{loc}}([0,\infty);W^{k,p}(\mt))$ for a.e. $s,t\in[0,\infty)$. Hence the claim follows from the equality of laws of $\bfU$ and $\mathcal{S}_{s-t}\bfU$.
\end{proof}

\begin{Remark}
Note that in view of Corollary \ref{cor:001} the stationarity in the sense of Definition \ref{D1} implies the following almost everywhere version of Definition \ref{D2}: if $\bfU$ is stationary on $L^q_{\mathrm{loc}}([0,\infty);W^{k,p}(\mt))$ in the sense of Definition \ref{D1} then the joint laws
$$\mathcal{L}(\bfU(t_1+\tau),\dots, \bfU(t_n+\tau)),\quad \mathcal{L}(\bfU(t_1),\dots, \bfU(t_n))$$
on $[W^{k,p}(\mt)]^n$ coincide for a.e. $\tau\geq0$, for a.e. $t_1,\dots,t_n\in [0,\infty)$.
\end{Remark}

\def\cprime{$'$} \def\ocirc#1{\ifmmode\setbox0=\hbox{$#1$}\dimen0=\ht0
  \advance\dimen0 by1pt\rlap{\hbox to\wd0{\hss\raise\dimen0
  \hbox{\hskip.2em$\scriptscriptstyle\circ$}\hss}}#1\else {\accent"17 #1}\fi}

\end{document}